\documentclass[12pt]{amsart}
\usepackage{amssymb}
\usepackage{amsmath}
\usepackage{amsthm}
\usepackage{multirow}
\newtheorem{theorem}{Theorem}[section]
\newtheorem{lemma}[theorem]{Lemma}

\newtheorem{corollary}[theorem]{Corollary}
\newtheorem{proposition}[theorem]{Proposition}

\newtheorem{lem-def}[theorem]{Lemma-Definition}
\DeclareRobustCommand\longtwoheadrightarrow
{\relbar\joinrel\twoheadrightarrow}

\newcommand{\hooklongrightarrow}{\lhook\joinrel\longrightarrow}
\newcommand{\I}{\mathbb I}
\newcommand{\R}{\mathbb R}

\newcommand{\N}{\mathbb N}
\newcommand{\Z}{\mathbb Z}
\newcommand{\Q}{\mathbb Q}

\def\op{\operatorname}

\def\as#1{\renewcommand\arraystretch{#1}}
\def\bs{\vskip.5cm}

\def\cc{{\,\mathcal C}}
\def\cci{{\,\mathcal C}\infty}

\def\cmpl{^{\mbox{\tiny $\op{cmpl}$}}}

\def\cut{{\op{cut}}}
\def\cutcl#1{\left[#1\right]_\cut}
\def\cuts{\op{Cuts}}
\def\cv{\op{Conv}}

\def\cofin{\op{cofin}}
\def\coini{\op{coini}}
\def\d{\Delta}
\def\defn{\noindent{\bf Definition. }}

\def\e{\medskip}
\def\ep{\epsilon}

\def\g{\Gamma}

\def\gbc{\g_{\op{bc}}}
\def\gd{\g(D)}
\def\gen#1{\big\langle\, {#1} \,\big\rangle}

\def\gm{\g_\mu}
\def\gn{\g_\nu}
\def\gnbc{\g_{\op{nbc}}}

\def\gpm{\g^\pm}
\def\gq{\g_\Q}

\def\gsme{\g_{\op{sme}}}

\def\ha{H(a)}

\def\hg{\mathbf{H}(\g)}
\def\hra{\hooklongrightarrow}

\def\hq{\mathbf{H}(\gq)}
\def\hk{\hookrightarrow}
\def\imp{\ \Longrightarrow\ }

\def\inii{\op{Init}(I)}

\def\iso{\lower.3ex\hbox{\as{.08}$\begin{array}{c}\lra\\\mbox{\tiny $\sim\,$}\end{array}$}}
\def\ism{\lower.3ex\hbox{\as{.08}$\begin{array}{c}\,\to\\\mbox{\tiny $\sim\,$}\end{array}$}}
\def\k{\op{Ker}}
\def\ka{\kappa}

\def\kp{\op{KP}}

\def\kpm{\op{KP}(\mu)}
\def\kx{K[x]}
\def\la{\lambda}
\def\La{\Lambda}

\def\lg{l\raise.6ex\hbox to.2em{\hss.\hss}l}
\def\ll{\mathcal{L}}
\def\lra{\,\longrightarrow\,}
\def\lx{\operatorname{lex}}

\def\md#1{\; \mbox{\rm(mod }{#1})}

\def\opp{^{\op{opp}}}

\def\p{\mathfrak{p}}

\def\pcv{\op{Prin}}
\def\ppr{^{\mbox{\tiny $\op{pr}$}}}

\def\rk{\op{rk}}
\def\rii{\R^\I_{\lx}}
\def\rll{\R_{\op{sme}}}
\def\rlex{\R^I_{\lx}}

\def\rrk{\op{rr}}

\def\sii{\ \Longleftrightarrow\ }
\def\Sme{{\op{Sme}}}
\def\sme{{\op{sme}}}

\def\ss{\mathcal{S}}

\def\supp{\op{supp}}
\def\sval{\op{sv}}

\def\tea{t_{\ep,a}}

\def\tsim{\mathcal{T}_\sim}
\def\tz{\mathcal{T}}
     
\def\val{\op{val}}

\def\xx{\mathcal{X}}

\title[Cuts and small extensions]{Cuts and small extensions of abelian ordered groups}

\makeatletter
\@namedef{subjclassname@2010}{%
  \textup{2010} Mathematics Subject Classification}

\thanks{Partially supported by grant PID2020-116542GB-I00 from the Spanish Research Agency}

\author[Kuhlmann]{Franz-Viktor Kuhlmann}
\address{}
\email{}

\author[Nart]{Enric Nart}
\address{Departament de Matem\`{a}tiques,         Universitat Aut\`{o}noma de Barcelona,         Edifici C, E-08193 Bellaterra, Barcelona, Catalonia}
\email{nart@mat.uab.cat}
\date{September 26, 2021}

\begin{document}

\begin{abstract}
We classify cuts in (totally) ordered abelian groups $\g$ and
compute the coinitiality and cofinality of all cuts in case $\g$ is divisible, in terms of 
data intrinsically associated to the invariance group of the cut. We relate cuts with small extensions of $\g$ in a natural way, which leads to an explicit construction of a totally ordered real vector space containing realizations of all cuts. This construction is applied to the problem of classifying all extensions of the valuation from a given valued field $K$ to the rational function field $K(x)$.
\end{abstract}

\maketitle

\section*{Introduction}

Cuts in abelian ordered groups (or ordered fields) are a very useful tool to attack certain problems in several disciplines where ordered structures play a significant role. Many of
them are briefly discussed in the survey \cite{KuhlCuts2}, with the most important of them
arising in the theory of ordered and valued fields.

For the basic notions we use, see Sections \ref{sectbagr} and \ref{sectcuts}.
The aim of this paper is to compute the coinitiality and cofinality of all cuts in a divisible totally ordered abelian group, in terms of data intrinsically associated to the  invariance group of the cut (see \cite{KuhlCuts2} for the origins of this notion).

Let $\g$ be a divisible totally ordered abelian group. 
For every cut $D$ in $\g$, let $H_D\subset\g$ be the invariance group of $D$. The cut $D$ is said to be a \emph{ball cut} if $D/H_D$ is a principal cut in the quotient group $\g/H_D$.


In Section \ref{secSmall}, we exhibit a natural relationship between cuts and \emph{small extensions} of $\g$.  
An order-preserving extension $\g\hk\La$ of ordered groups is said to be \emph{small} if $\La/\g$ is a cyclic group.
Every cut $D$ in $\g$ determines a natural small extension $\g\subset \gd$ where the cut has a realization,  and we show that ball cuts are characterized by the fact that the extension $\g\subset\gd$ increases the rank (Theorem \ref{ballsIncrease}). 

In Section \ref{secRealVS}, we use Hahn's embedding theorem to construct a totally ordered  real vector space containing all small extensions of $\g$, up to order-preserving $\g$-isomorphism. The subspace containing all rank-preserving small extensions is clearly distinguished.
These real vector spaces are explicit enough to facilitate the computation of the coinitiality and cofinality of all non-ball cuts (Theorem \ref{CofNB}) and all ball cuts (Theorem \ref{CofB}). Let us mention that for the case of ordered fields $K$ the similar task of constructing a universal extension that contains the rational function fields 
$K(x_1,\ldots,x_n)$ with all possible extensions of the ordering was dealt with in \cite{KKMZ}. 



In Section \ref{secVals}, we apply these real models containing all small extensions, to the problem of classifying valuations on polynomial rings $K[x]$, up to equivalence. As these 
uniquely determine their extensions to the quotient field $K(x)$, this is a contribution to the ongoing research on the extensions of the valuation from a given valued field $K$ to the rational function field $K(x)$.

All methods of the paper are developed for arbitrary ordered abelian groups, and the condition of divisibility is used only when this is strictly necessary.




\section{Background on ordered abelian groups}              \label{sectbagr}
In this section, we collect some basic facts on ordered abelian groups, mainly extracted from \cite{KuhlBook} and \cite{Rib}.

\subsection{Ordered sets and cuts}

Throughout the paper, an \emph{ordered set} will be a set equipped with a total ordering. We agree that $0\not\in\N$.\e

\noindent{\bf Notation. }Let $I,\,J$ be ordered sets.

\begin{itemize}
\item $I\infty$ is the ordered set obtained by adding a (new) maximal element, which is formally denoted as $\infty$. 
\item $I\opp$ is the ordered set obtained by reversing the ordering of $I$.
\item For all $i\in I$, we denote $I_{< i}=\{j\in I\mid j< i\}\subset I_{\le i}=\{j\in I\mid j\le i\}$. 

We attribute a similar meaning to $I_{>i}\subset I_{\ge i}$. 
\item For $S,T\subset I$ and $i\in I$, the following expressions
$$i<S,\ \quad i>S,\ \quad i\le S,\ \quad i\ge S,\ \quad S<T,\ \quad S\le T$$
mean that the corresponding inequality holds for all $s\in S$ and all $t\in T$. 
\item For $S\subset I$, we denote by $S^c=I\setminus S$ the complementary subset. 
\item $I+J$ is the disjoint union $I\sqcup J$ with the total ordering which respects the orderings of $I$ and $J$ and satisfies $I<J$.
\end{itemize}

A mapping $\iota\colon I\to J$  is an \emph{embedding} if it strictly preserves the ordering.
We also say that $\iota\colon I\to J$ is an \emph{extension} of $I$. 

An \emph{isomorphism} of ordered sets is an onto embedding.
The \emph{order-type} of an ordered set is the class of this set up to isomorphism. \e

An \emph{initial segment} of $I$ is a subset $S\subset I$ such that $I_{\le i}\subset S$ for all $i\in S$.
A \emph{final segment} of $I$ is a subset $T\subset I$ such that $I_{\ge i}\subset T$ for all $i\in T$.

On the set $\inii$ of all initial segments of $I$ we consider the ordering determined by ascending inclusion. It has a minimal and a maximal element:
$$
\emptyset=\min(\inii),\qquad I=\max(\inii).
$$

A \emph{cut} in $I$ is a pair $D=(D^L,D^R)$ of subsets such that $$D^L< D^R\quad\mbox{ and }\quad D^L\cup D^R=I.$$ Clearly, $D^L$ is an initial segment of $I$ and $D^R=(D^L)^c$ is a final segment. 

If $\cuts(I)$ denotes the set of all cuts in $I$, we have an isomorphism of ordered sets
$$
\inii\lra\cuts(I),\qquad S  \longmapsto (S,S^c).
$$
In particular, $\cuts(I)$ admits a minimal element $(\emptyset,I) $and a maximal element $(I,\emptyset)$, which are called the  \emph{improper cuts}.

For all $M\subset I$, we denote by  $M^+$, $M^-$ the cuts determined by the initial segments
\[
M^+=\{i\in I\mid \exists m\in M: i\leq m\},\qquad M^-=\{i\in I\mid i<M\}.
\]

If $M=\{i\}$, we will write $i^+=I_{\le i}$ instead of $\{i\}^+$ and $i^-=I_{<i}$ instead of $\{i\}^-$. These cuts are said to be 
\emph{principal}. 

Hence, a cut $D$ is principal if  either $D^L$ has a maximal element, or $D^R$ has a minimal element. A cut for which these two conditions hold simultaneously is called a \emph{gap}, or a \emph{jump} according to different authors.

If $I\hk J$ is an extension of ordered sets and $j \in J$ satisfies
$D^L\leq j\leq D^R$, then we will say that \emph{$j$
realizes the cut $D$} in $J$.

Every cut $D$ has associated an important invariant $$(\ka(D),\la(D)),$$ which is a pair of regular cardinals. The cardinal $\ka(D)$ is the cofinality of $D^L$ (least cardinality of any cofinal subset). Dually, $\la(D)$ is the coinitiality of $D^R$ (least cardinality of any coinitial subset).

\subsection{Ordered groups}

An \emph{ordered group} $(\g,<)$ is an (additive) abelian group $\g$ equipped with a total ordering $<$, which is compatible with the group structure.

For all $a\in\g$, we denote $|a|=\max(a,-a)$.

An ordered group has no torsion. In fact, all nonzero $a\in\g$  satisfy $$n|a|\ge|a|>0,\quad \forall\,n\in\N.$$

An \emph{embedding/extension/isomorphism} of ordered groups is a group homomorphism which is simultaneously an embedding/extension/isomorphism  of ordered sets.

A basic example of ordered group is $\R^n_{\lx}$, the additive group $(\R^n,+)$ equipped with the lexicographical ordering. 
Also, any subgroup of an ordered group inherits the structure of an ordered group.


The \emph{divisible hull} of an ordered group $\g$ is the group $$\gq=\g\otimes_\Z\Q,$$ which inherits a natural ordering determined by the condition
$$
a\otimes(1/n)<b\otimes (1/m)\sii ma<nb,
$$
for all $n,m\in\N$ and all $a,b\in\g$.

Since $\g$ has no torsion, it may be embedded in a unique way into $\gq$ as an ordered group. 
The divisible hull of $\g$ is the minimal divisible extension of $\g$.

\begin{lemma}\label{MInDiv}
For any embedding $\iota\colon\g\hk \La$ into a divisible ordered group $\La$, there exists a unique embedding of $\gq$ into $\La$ such that $\iota$ coincides with the composition $\g\hk \gq\hk\La$. 
\end{lemma}


Given a subgroup $H\subset \g$, the quotient $\g/H$ inherits a structure of ordered group if and only if $H$ is a \emph{convex} subgroup; that is, 
$$
h\in H \imp  \left\{x\in \g\mid |x|\le|h|\right\}\subset H.
$$
In this case, we may define an ordering in $\g/H$ by:
$$
a+H<b+H \sii a+H\ne b+H \ \mbox{ and }\ a<b.
$$

The notation $a+H<b+H$ is compatible with the natural meaning of such an inequality for arbitrary subsets of $\g$. 
The convex subgroups of $\g/H$ are in 1-1 correspondence with the convex subgroups of $\g$ containing $H$. 

\begin{lemma}
Let $f\colon \g\to\d$ be an order-preserving group homomorphim between two ordered groups. Then, $\k(f)$ is a convex subgroup of $\g$ and the natural isomorphism between $\g/\k(f)$ and $f(\g)$ is order-preserving too.    
\end{lemma}

\begin{lemma}
The convex subgroups of $\g$ are totally ordered by inclusion.
\end{lemma}

\begin{proof}
Let $H$, $H'$ be convex subgroups such that there exists $a\in H\setminus H'$.

Then, for all $b\in H'$ we must have $|b|<|a|$, so that $H'\subset H$. 
\end{proof}

\noindent{\bf Definition. }Let $\cc=\cv(\g)$ be the ordered set of all {\bf proper} convex subgroups, ordered by {\bf ascending} inclusion
$$
\{0\}\subset \cdots \subset H\subset \cdots  \subsetneq \g
$$ 
The order-type of $\cc$ is called the $\emph{rank}$ of $\g$, and is denoted $\rk(\g)$.

We identify $\cc\infty$ with the ordered set of all convex subgroups of $\g$, by letting $\infty$ represent the whole group $\g$. \e

\noindent{\bf Examples. }
\begin{itemize}
\item $\rk(\Z)=\rk(\Q)=\rk(\R)=1$. \e

\item $\rk(\R^n_{\lx})=n$. The sequence of convex subgroups is
$$
\{0_{\R^n}\}\ \subset \ \cdots\ \subset \{0\}^k\times\R^{n-k}_{\lx}\ \subset\ \cdots\ \subset \ \R^n_{\lx},
$$
\end{itemize}


For all $a\in\g$, we denote by $\ha$ the convex subgroup of $\g$ generated by $a$:
$$
\ha=\left\{b\in \g\,\mid\, |b|\le n|a| \ \mbox{ for some } n\in\N \right\}.
$$ 


These convex subgroups $\ha$ are said to be \emph{principal}. \e

\noindent{\bf Definition. }Let $I=\pcv(\g)$ be the ordered set of {\bf nonzero} convex principal subgroups of $\g$, ordered by {\bf descending} inclusion.

\e

We identify $I\infty$ with a set of indices parametrizing all principal convex subgroups of $\g$. For all $i\in I$ we denote by $H_i$ the corresponding  principal convex subgroup. We agree that $H_\infty=\{0\}$. According to our convention, for all $i,j\in I\infty$, we have
$$
i<j \sii H_i\supsetneq H_j.
$$

\begin{lemma}\label{allH}
Every convex subgroup $H\subset \g$ satisfies $H=\bigcup_{i\in I,\,H_i\subset H}H_i$.
\end{lemma}

\begin{proof}
For all $a\in H$, the principal convex subgroup $H(a)$ is contained in $H$.
\end{proof}

\begin{corollary}
If $I$ is well-ordered, then all convex subgroups are principal. 
\end{corollary}

\begin{proof}
For any convex subgroup  $H$, the subset $\{i\in I\mid H_i\subset H\}\subset I$ has a minimal element $i_0$. By Lemma \ref{allH}, $H=H_{i_0}$.
\end{proof}

\subsection{Archimedean classes  and natural valuation}\label{secHahn}

Two elements $a,b\in\g$ are \emph{archimedean equivalent} if  they generate the same convex subgroup: $\ha=H(b)$.
This defines an equivalence relation on $\g$, whose quotient set is $I\infty$.

We say that $\g$ is \emph{archimedean} if all nonzero elements are archimedean equivalent; that is, if $\{0\}$ and $\g$ are the only convex subgroups of $\g$.   

\begin{proposition}\label{rk1}
If $\g$ is nontrivial, the following conditions are equivalent.
\begin{enumerate}
\item $\g$ is archimedian.
\item $\g$ has rank one.
\item $\g$ is order-isomorphic to a subgroup of $\R$. 
\end{enumerate}
\end{proposition}

\begin{proof}
Conditions (1) and (2) are obviously equivalent. Also, clearly (3) implies (1).

Finally, if $\g$ is archimedian, the choice of any positive $b\in \g$ determines a unique embedding $\g\hk\R$ of ordered groups, such that $b\mapsto 1$. Indeed, any $a\in\g$ is mapped to the real number determined by the initial segment of the rational numbers $m/n$ such that $mb\le na$. Thus, (1) implies (3).
\end{proof}

The \emph{natural valuation} on $\g$ is the mapping 
$$
\val:=\val_\g\colon \g\lra I\infty,\qquad  a\longmapsto \val(a)=H(a).  
$$

This mapping satisfies the following two properties, for all $a,b\in\g$:\e

(VAL0)\quad $\val(a)=\infty \sii a=0$.\e

(VAL1)\quad $\val(a-b)\ge\min\{\val(a),\val(b)\}$, \ and equality holds if $\ \val(a)\ne\val(b)$. \e

We may consider balls in $\g$ with center an element $a\in\g$ and radius $i\in I$:
$$
\as{1.2}
\begin{array}{l}
B_i(a)=\left\{x\in\g\mid \val(x-a)\ge i\right\}=a+H_i,\\
B^o_i(a)=\left\{x\in\g\mid \val(x-a)>i\right\}=a+H^*_i,
\end{array}
$$
where $H_i^*$ is the union of all principal convex subgroups strictly contained in $H_i$.

The natural valuation determines an ultrametric topology on $\g$ by taking the set $\left\{B^o_i(a)\mid i\in I\right\}$ as a fundamental system of open neighbourhoods of $a$.

For all $i\in I$, the subgroup $$H_i^*=B^o_i(0)\subsetneq H_i$$ 
is a convex subgroup (not necessarily principal) which
is the immediate predecessor of $H_i$ in the ordered set $\cc\infty$. 
In particular, the quotient
$$
C_i=C_i(\g)=H_i/H_i^*
$$
 is an ordered group of rank one, which is called the $i$-th \emph{component} of $\g$.

The \emph{skeleton} of $\g$ is the pair $\left(I,(C_i)_{i\in I}\right)$.



\subsection{Relationship between $\cc$ and $I$}



The ordered sets $I$ and $\cc\infty$ are determined one by the other.

\begin{lemma}\label{ConvI}
The set $I$ is the subset of $\cc\infty$ formed by all elements admitting an immediate predecessor.
\end{lemma}

\begin{proof}
Any nonzero principal convex subgroup $H_i$ has an immediate predecessor $H^*_i$.

Conversely, if $H'\subsetneq H$ is an immediate predecessor of a convex subgroup $H$, then $H$ is the principal subgroup generated by any $a\in H\setminus H'$.
\end{proof}

To any cut $(S,S^c)$ in $I$ we may associate the convex subgroup
$$
H_S=\bigcup\nolimits_{i\in I,\,i\not\in S}H_i.
$$

\begin{lemma}\label{IConv}
The assignment $(S,S^c)\mapsto H_S$ yields an isomorphism of ordered sets:
$$
\cuts(I)\opp\lra \cc\infty.
$$

The  inverse isomorphism maps any $H\in\cc\infty$ to the cut $\left(\val\left(\g\setminus H\right),\val\left(H\right)\right)$.
\end{lemma}

\begin{proof}
If $S\subsetneq T$ are initial segments of $I$, then $H_T\subset H^*_i\subsetneq H_i\subset H_S$, for all
$i\in T\setminus S$. Thus, the mapping $(S,S^c)\mapsto H_S$ is an embedding of ordered sets. 
Also, it is an onto map because $H=H_{\val\left(\g\setminus H\right)}$ by Lemma \ref{allH}.
\end{proof}

\begin{corollary}\label{quotient}
For $H\in\cc$, let $S=\val\left(\g\setminus H\right)$ be the initial segment of $I$ such that $H=H_S$. Then, the assignment $H'\mapsto H'/H$ induces isomorphisms of ordered sets:
$$
\cc_{\ge H}\simeq \cv(\g/H),\qquad \val\left(\g\setminus H\right)\simeq \pcv(\g/H).
$$
\end{corollary}

\subsection{Immediate extensions}\label{subsecImmed}

\begin{lemma}\label{rkEmbed}
Every extension $\g\hk\La$ of ordered groups induces natural embeddings of ordered sets:
$$
\cv(\g)\hra\cv(\La),\quad H\mapsto \;\overline{H}\mbox{ \ convex subgroup of $\La$  generated by }H.
$$
$$
\pcv(\g)\stackrel{\iota}\hra\pcv(\La),\qquad H_i=\val_\g(a)\mapsto H_{\iota(i)}=\overline{H_i}=\val_\La(\iota(a)).
$$  
Moreover, the embedding $H_i\hk \overline{H_i}$ induces an embedding $C_i(\g)\hk C_{\iota(i)}(\La)$. 
\end{lemma}

\noindent{\bf Definition. }The extension $\g\hk\La$ is \emph{immediate} if it preserves the skeleton. That is, it induces an isomorphism $\pcv(\g)\simeq \pcv(\La)$ of ordered sets, and isomorphisms $C_i(\g)\simeq C_{\iota(i)}(\La)$ between all the components.\e


It is possible to construct ordered groups with a prescribed skeleton.

Let $I$ be an arbitrary ordered set and $(C_i)_{i\in I}$ a family of ordered groups of rank one, parametrized by $I$. 
Their \emph{Hahn sum} is the direct sum  $\bigoplus\nolimits_{i\in I}C_i$  equipped with the lexicographical order.

For any element $x=(x_i)_{i\in I}\in\prod_{i\in I}C_i$, the \emph{support} of $x$ is the subset 
$$\op{supp}(x)=\left\{i\in I\mid x_i\ne0\right\}\subset I.$$

The \emph{Hahn product} is  defined as the subgroup $$\lower.5ex\hbox{\mbox{\Large$\mathbf{H}$}}{\lower1.3ex\hbox{\mbox{\tiny${i\!\!\in\!\! I}$}}}\,C_i\ \subset\ \prod\nolimits_{i\in I}C_i$$  formed by all elements whose support is a well-ordered subset of $I$, with respect to the ordering induced by that of $I$.
It is easy to check that it makes sense to consider the lexicographical ordering on this subgroup.

The Hahn product is an immediate extension of the Hahn sum because both ordered groups have skeleton $\left(I,(C_i)_{i\in I}\right)$.
More precisely, in both cases the principal subgroups are parametrized by $I$ via: 
$$
H_i=\{x=(x_j)_{j\in I}\mid x_j=0\mbox{ for all }j<i\},\qquad i\in I.
$$
Moreover, the projection 
$
H_i\to C_i$, sending $x\mapsto x_i$,
induces an isomorphism of ordered groups between $H_i/H_i^*$ and $C_i$.

The convex subgroup $H_S$ determined by an initial segment $S\in \inii$ is:
$$
H_S=\{(x_j)_{j\in I}\mid x_j=0\mbox{ for all }j\in S\}.
$$

If $C_i=C$ for all $i\in I$, then we use the notation
$$
C^{(I)}\subset\, C^I_{\lx}\subset\, C^I,
$$
for the Hahn sum, Hahn product, and cartesian product, respectively.


The following result is known as \emph{Hahn's embedding theorem} \cite[Sec. A]{Rib}.

\begin{theorem}\label{hahn}
 Every divisible ordered group $\g$ admits an immediate embedding in the Hahn product determined by the skeleton of $\g$.
\end{theorem}


Let us quickly review the proof of Hahn's theorem based on the following result of Banaschewski, which makes use of Zorn's lemma.  \e

\noindent{\bf Banaschewski's lemma. }{\it Let $V$ be a vector space over a field $K$. Consider a non-empty set $ \ss\subset\op{Subsp}(V)$ of subspaces of $V$.
 Then, there exists a mapping
$$
(\ \,)\cmpl\colon \ss\lra\ \op{Subsp}(V)
$$
satisfying the following properties:
\begin{enumerate}
\item[(i)] $\ V=W\oplus W\cmpl$, \ for all $W\in\ss$.
\item[(ii)] $\ U\subset W\;\imp\,U\cmpl\supset W\cmpl$, \ for all $\,U,W\in\ss$. 
\end{enumerate}}\e

Now, let $\g$ be a divisible ordered abelian group 
with skeleton $\left(I,(C_i)_{i\in I}\right)$. The group $\g$ and its components $C_i$ have a natural structure of $\Q$-vector spaces; hence, the Hahn product
is a $\Q$-vector space too.  Hahn's embedding
$$
\varphi\colon \g\hra\hg:=\lower.5ex\hbox{\mbox{\Large$\mathbf{H}$}}{\lower1.3ex\hbox{\mbox{\tiny${i\!\!\in\!\! I}$}}}\,C_i
$$
is necessarily a $\Q$-linear mapping.

By Banaschewski's lemma, we may choose complementary $\Q$-subspaces:
$$
\g=H\oplus H\cmpl,\quad\mbox{for all }H\in \cv(\g), 
$$
with the coherent behaviour with respect to inclusions indicated in condition (ii). 

For all $i\in I$ we have $\Q$-linear projections:
$$
\g\longtwoheadrightarrow H_i\longtwoheadrightarrow C_i,\qquad a\longmapsto a_i. 
$$
The projection $\g\twoheadrightarrow H_i$ depends on the choice of the subspace $H_i\cmpl$. The projection $H_i\twoheadrightarrow C_i=H_i/H_i^*$ is the canonical quotient mapping.

In this way, we obtain an injective group homomorphism:
$$
\varphi\colon\g\lra \prod_{i\in I}C_i,\qquad a\ \longmapsto\ \left(a_i\right)_{i\in I}.
$$
Indeed, if $a\in\g$ is non-zero, the principal subgroup $H(a)$ generated by $a$ is non-zero too; thus,  $H(a)=H_i$ for some $i\in I$. The element $a_i\in C_i$ is the class of $a$ modulo $H_i^*$.
Since $a$ generates $H_i$, we have  $a_i\ne0$.

The proof of Hahn's theorem ends by checking that $\varphi(\g)\subset\hg$ and $\varphi$ preserves the ordering \cite[Sec. A]{Rib}.\e

Let us recall some nice properties of this embedding $\varphi$. 
For all $S\in\inii$, consider the well-known process of \emph{truncation by} $S$:
$$
\hg\lra\hg,\qquad x=(x_i)_{i\in I}\longmapsto x_S=(y_i)_{i\in I},
$$
where $y_i=x_i$ for all $i\in S$, and $y_i=0$ otherwise. 

\begin{lemma}\label{hahnSpecial}
The image  $\varphi(\g)$ of Hahn's embedding contains the truncations $\varphi(a)_S$, for all $a\in\g$, $S\in\inii$. In particular, $\varphi(\g)$ contains the Hahn sum $\bigoplus_{i\in I}C_i$.
\end{lemma}

\begin{proof}

Take $a\in\g$, $S\in\inii$. We can write in a unique form:
$$
a=c+b,\quad c\in H_S,\quad b\in H_S\cmpl.
$$
Let us show that $\varphi(a)_S=\varphi(b)\in\varphi(\g)$. Indeed, for $i\in S$ we have $H_S\subset H_i^*\subsetneq H_i$. Thus, the decompositions of $a$ and $b$ over $H_i$ are:
$$
b=u+v,\qquad a=(c+u)+v,\qquad u,c\in H_i,\ v\in H_i\cmpl.
$$
Since $c\in H_S\subset H_i^*$, we have $a_i=(c+u)\md{H_i^*}=u\md{H_i^*}=b_i$.

For $i\not\in S$ we have $H_i\subset H_S$, so that $b\in H_S\cmpl\subset H_i\cmpl$. Thus, $b_i=0$. This proves that $\varphi(\g)$ contains all truncations.

Take now $i\in I$ and $q\in C_i$. Choose any $a\in H_i$ such that $q=a\md{H_i^*}$. 
For all $j<i$, we have $H_j\supsetneq H_i$. Thus, $H_j^*\supseteq H_i$ and $a_j=0$. Hence, for $S=I_{\le i}$, the projection $\varphi(a)_S$ has $i$-th coordinate $q$, and all other Hahn coordinates vanish. Since $\varphi(a)_S\in\varphi(\g)$, this proves  that $\varphi(\g)$ contains the Hahn sum.
\end{proof}

\section{Cuts in ordered abelian groups}                    \label{sectcuts}
In this section we study basic properties of cuts in an ordered abelian group $\g$. Most of the material has been taken from \cite{KuhlCuts}.

We keep with the notation $\cc=\cv(\g)$, $\,I=\pcv(\g)$,\, used in the last section.

Any cut $(D,E)$ in $\g$ admits the operations of \emph{shifting} by an element $a\in\g$ and \emph{multiplication} by $-1$. They are defined in the obvious way: 
$$
(D,E)+a=(D+a,E+a),\qquad -(D,E)=(-E,-D).
$$
These operations are defined for the improper cuts too, if we agree that 
$$-\emptyset=\emptyset,\qquad \emptyset+a=\emptyset \  \mbox{ for all } \ a\in \g.
$$

\subsection{Invariance group of a cut}

To every subset $D\subset \g$ we may associate the following invariance subgroup:
$$
H(D)=\left\{h\in\g\mid D+h=D\right\}\subset\g.
$$
Let $E=\g\setminus D$ be the complementary subset of $D$.
For all $h\in\g$, we clearly have $\g\setminus(D+h)=E+h$; therefore, $$D+h=D \quad\sii\quad E+h=E.$$

In particular, to every cut $(D,E)$ in $\g$, we may associate its \emph{invariance group}: 
$$
H(D,E):=H(D)=H(E).
$$

Now, since $D$ is an initial segment of $\g$, the subgroup $H(D)\subset\g$ is convex. Indeed, take a positive $h\in H(D)$ and an element $a\in\g$ such that $0<a<h$.  We have necessarily $a\in H(D)$, because
$$
D=D-h\subset D+a\subset D+h=D.
$$

\begin{lemma}\label{propertiesHD}
Let $(D,E)$ be a cut in $\g$, and let  $H\subset\g$ be a convex subgroup.
\begin{enumerate}
\item[(i)\;] \ $E-D=\left(\g_{>0}\right)\setminus H(D)$.
\item[(ii)\,] \ We have $H\subset H(D)$ if and only if $(D/H,E/H)$  is a cut in $\g/H$. In this case, the invariance group of this cut is $H(D)/H$.
\item[(iii)] \ $H(-D)=H(D)$ and $H(D+a)=H(D)$ for all $a\in\g$.
\end{enumerate}
\end{lemma}

\begin{proof}
Item (iii) is obvious. Let us prove (i). Suppose that $h$ is positive. Then,
$$h\not\in H(D)\sii (D+h)\cap E\ne \emptyset\sii h\in E-D.$$
Item (ii) follows from: $(D/H)\cap(E/H)=\emptyset\sii  D+H=D\sii H\subset H(D)$.
\end{proof}

Let $S=\{i\in I\mid H_i\supsetneq H(D)\}$ be the initial segment of $I$ canonically associated to $H(D)$ in Lemma \ref{IConv}. Let $H'\in\cci$ be the unique possible candidate to be an immediate successor of $H(D)$ in $\cci$. That is,  
\begin{equation}\label{Hprima}
\bigcup_{i\not\in S}H_i=H(D)\subset H':=\bigcap_{i\in S}H_i.
\end{equation}


\begin{lemma}\label{convGap}
For any cut $(D,E)$ in $\g$, the following conditions are equivalent.
\begin{enumerate}
\item[(a)] \ The principal cut $H(D)^+$ in $\cc\infty$ is a gap.
\item[(b)] \ There exists a maximal element in $S$.
\item[(c)] \ $H(D)\subsetneq H'$.
\end{enumerate}
In this case, $H(D)$ is the immediate predecessor of the principal convex subgroup $H'=H_{i_{\max}}$, where $i_{\max}=\max(S)$.
\end{lemma}

\begin{proof}
If $H(D)^+$ is a gap in $\cc\infty$, there exists an immediate successor of $H(D)$ in $\cc\infty$. This successor must be principal by Lemma \ref{ConvI}; thus, it coincides wih $H'$. This shows that (a) implies (b).

If there exists $i_{\max}=\max(S)$, then obviously $H'=H_{i_{\max}}$. Hence,
$$H(D)=\bigcup_{i\not\in S}H_i=H_i^*\subsetneq H_i.$$ Thus, (b) implies (c).

Finally, if $H(D)\subsetneq H'$, then $H'$ is the minimal element in $\left(\cc\infty\right)_{>H(D)}$. Therefore, $H(D)^+$ is a gap.  Thus, (c) implies (a).
\end{proof}\e

\noindent{\bf Definition. }If the conditions of Lemma \ref{convGap} hold, we say that the cut $(D,E)$ \emph{has a convexity gap}. \e

The improper cuts determined by $D=\emptyset,\,\g$ have both $H(D)=\g$ and $S=\emptyset$. Hence, they do not have a convexity gap. 

\subsection{Covariance subgroups of a subset of $\g$}

Let us associate a couple of \emph{covariance} convex subgroups to every subset $D\subset\g$.

For all $d\in D$, let $V_f(d)=V_f(D,d)$ be the convex subgroup generated by the set 
$$D_{\ge d}-d=\left\{d'-d\mid d'\in D,\ d'\ge d\right\}.$$ If $d<d'$, then $V_f(d)\supset V_f(d')$, because
$$
e\in D_{\ge d'}\ \imp\ 0\le e-d'<e-d\in D_{\ge d}-d\ \imp\ e-d'\in V_f(d).
$$
\vskip0.2cm 

\noindent{\bf Definition. }The \emph{final covariance} group $V_f(D)$, and \emph{initial covariance} group $V_i(D)$ of $D$ are defined as
$$
V_f(D)=\bigcap_{d\in D}V_f(d),\qquad V_i(D):=V_f(D\opp)=\bigcap_{d\in D}V_f(D\opp,d).
$$

The group $V_f(D)$ is said to be \emph{stable} if $V_f(D)=V_f(d)$ for some $d\in D$.

The group $V_i(D)$ is said to be \emph{stable} if $V_i(D)=V_f(D\opp,d)$ for some $d\in D$.

A covariance group which is not stable is said to be \emph{unstable}.\bs

The following basic properties of these groups are easy to check.\e

$\bullet$\quad If there exists $d_M=\max(D)$, then $V_f(D)=V_f(d_M)=0$ is stable. 

If there exists $d_m=\min(D)$, then $V_i(D)=V_f(D\opp,d_m)=0$ is stable.\e

$\bullet$\quad Let $\op{ini}(D)$, $\op{fin}(D)$ be the initial segment and final segment of $\g$ generated by $D$, respectively. Then,
$V_f(D)=V_f(\op{ini}(D))$ and $V_i(D)=V_i(\op{fin}(D))$.\e

$\bullet$\quad $V_f(D\opp,d)=V_f(-D,-d)$, \ for all $d\in D$.\e

$\bullet$\quad $V_f(-D)=V_i(D)$, \ $V_i(-D)=V_f(D)$.\e

$\bullet$\quad $V_f(D+a)=V_f(D)$, \ $V_i(D+a)=V_i(D)$, \,for all $a\in\g$.\e

$\bullet$\quad $V_f(\g)=V_i(\g)=\g$ is stable. We agree that $V_f(\emptyset)=V_i(\emptyset)=\g$ is unstable.

\subsection*{Covariance subgroups of a cut}
To every cut $(D,E)$ in $\g$ we may associate the couple of convex subgroups
$$
V_f(D,E):=V_f(D),\quad V_i(D,E):=V_i(E).
$$

The invariance group of a cut was a symmetric object, but these covariance subgroups of a cut do not always coincide. The most simple example is the cut $0^+$ in $\g=\Q$, for which $V_f(0^+)=0$ and $V_i(0^+)=\Q$. 
However, we have in full generality
$$
V_f(D,E)=V_i(-(D,E)),\qquad V_i(D,E)=V_f(-(D,E)).
$$

\begin{lemma}\label{breadth}
Let $(D,E)$ be a cut in $\g$, and denote $H=H(D)$. Then,
\begin{equation}\label{InTheMiddle}
H\subset V_f(D),\,V_i(E)\subset H', 
\end{equation}
where $H'$ is the convex subgroup defined in (\ref{Hprima}). 
In particular, either $V_f(D)=H$, or $V_f(D)=H'$, and an analogous statement holds for $V_i(E)$.

Moreover, the covariance groups of the cut $\left(D/H,E/H\right)$ in $\g/H$ are $V_f(D/H)=V_f(D)/H$ and $V_i(E/H)=V_i(E)/H$.
\end{lemma}

\begin{proof}
For all $d\in D$, the condition $d+H\subset D$ implies $H\subset V_f(d)$. Therefore, $H\subset V_f(D)$.

Take any positive $a\in\g\setminus H$; Lemma \ref{propertiesHD}(i) shows the existence of $d\in D$ such that $d+a>D$. For all $d'\in D_{\ge d}$ we have $0\le d'-d<a$. Hence,
\begin{equation}\label{Vd}
V_f(D)\subset V_f(d)\subset H(a),
\end{equation} 
where $H(a)$ is the principal convex subgroup generated by $a$. Since this holds for all  $a\in\g\setminus H$, we deduce that  $V_f(D)\subset H'$.   

Now, since $H(-D)=H(D)=H$, and $V_f(D\opp)=V_f(-D)$, the above arguments show that $H\subset V_i(E)=V_f(D\opp)\subset H'$ as well. 
This ends the proof of (\ref{InTheMiddle}). 

The proof of the last statement is straightforward.
\end{proof}

\begin{lemma}\label{gap->stable}
If a cut $(D,E)$ in $\g$ has a convexity gap, then its covariance groups $V_f(D)$, $V_i(E)$ are both stable.
\end{lemma}

\begin{proof}Let $H=H(D)$ be the invariance group of the cut.
The group $V_f(D)$ is stable if and only if the final covariance group $V_f(D)/H$ of the cut $(D/H,E/H)$ is stable. Hence,  we may assume $H=\{0\}$. In this case,  $H'$ is a minimal nonzero convex subgroup. By Lemma \ref{ConvI}, $H'$ is a principal convex subgroup. 

If there exists $d_{\max}=\max(D)$, then $V_f(D)=V_{d_{\max}}=\{0\}$ is stable. 

Otherwise,
$V_f(d)\supset H'$ for all $d\in D$; hence, $V_f(D)\supset H'$ and this implies $V_f(D)=H'$ by Lemma \ref{breadth}. Let $a\in V_f(D)$ be a positive generator of $H'$ as a convex subgroup. Since $a\not\in H$,  the arguments in the proof of Lemma \ref{breadth} lead to the inclusions (\ref{Vd}) for some $d\in D$:
$$
V_f(D)\subset V_f(d)\subset H(a)=H'=V_f(D).
$$
Thus, $V_f(D)=V_f(d)$ is stable too.

Since the cut $-(D,E)$ has the same invariance group $H$, it has a convexity gap too. Thus, $V_i(E)=V_f(-(D,E))$ is stable.
\end{proof}

\noindent{\bf Definition. }A cut $(D,E)$ in $\g$ is said to be \emph{vertical} if $H(D)\subsetneq V_f(D)$. 
Otherwise, we say that $(D,E)$ is \emph{horizontal}. \e

By Lemma \ref{breadth}, all cuts which do not have a convexity gap are horizontal, because they satisfy
$$
H(D)=V_f(D)=V_i(E)=H'.
$$

The adjective \emph{vertical} is motivated by the applications to valuation theory. There is a certain procedure of limit augmentation of valuations, which involves increasing families $D$ of values in a certain ordered abelian group $\g$. The cut in $\g$  determined by the initial segment generated by $D$ is vertical if and only if the family $D$ is ``vertically bounded"  in the terminology of \cite[Sec. 4]{AFFGNR}.

\subsection{The six types of cuts}\label{subsec6types}
Let us study the possible types of cuts according to its convexity-gap character plus the fact of being horizontal/vertical, and their covariance groups being stable/unstable.\bs

\noindent{\bf Definition. }Let $(D,E)$ be a cut in $\g$, with invariance group $H$. We say that $(D,E)$ is a \emph{ball cut}  if the cut $(D/H,E/H)$ in the quotient group $\g/H$ is principal. \bs

Thus, a  ball cut in $\g$ takes the form 
$$
(a+H)^+ \quad\mbox{ \ or \ }\quad (a+H)^-,
$$
for some $a\in D$ and an arbitrary convex subgroup $H\subset\g$, which clearly becomes the invariance group of the cut.

For instance, the principal cuts and the improper cuts are ball cuts, with $H=\{0\}$ and $H=\g$, respectively. 

Note that $-(a+H)^+=(-a+H)^-$. Thus, the set of all ball cuts is closed under the ``multiplication by $-1$" operation.

For the non-ball cuts, the situation is quite rigid. The convexity-gap character determines all other properties of the cut.

\begin{lemma}Let $(D,E)$ be a non-ball cut in $\g$, with invariance group $H$.
\begin{enumerate}
\item If $(D,E)$ has a convexity gap, then $H\subsetneq V_f(D)=V_i(E)$ and the covariance groups are stable. 
\item If $(D,E)$ does not have a convexity gap, then $H=V_f(D)=V_i(E)$ and the covariance groups are unstable.
\end{enumerate}
\end{lemma}

\begin{proof}
All properties of the cut are preserved when we pass to the quotient $\g/H$. Thus, we may assume that $H=\{0\}$ and the cut is not principal. 

For all $d\in D$ there exists a positive $h\in\g$ such that $d+h\in D$. Since $h\in V_f(d)$, we see that $V_f(d)\ne\{0\}$ for all $d\in D$. The same argument shows that $V_f(E\opp,e)\ne\{0\}$ for all $e\in E$.

If the cut has a convexity gap, the covariance groups are stable by Lemma \ref{gap->stable}. Thus, 
$V_f(D)=V_f(d)\supsetneq H$ for some $d\in D$, and simultaneously
$$ 
V_i(E)=V_f(E\opp)=V_f(E\opp,e)\supsetneq H,\quad \mbox{for some }e\in E.
$$
Therefore, $V_f(D)=V_i(E)=H'$, by Lemma \ref{breadth}.

If the cut does not have a convexity gap, then $H=V_f(D)=V_i(E)=H'=\{0\}$. Hence $V_f(D)\ne V_f(d)$ for all $d\in D$ and $V_i(E)\ne V_f(E\opp,e)$ for all $e\in E$. Both covariance groups are unstable.
\end{proof}

On the other hand, since the analysis of the properties of ball cuts may be reduced to the analysis of principal cuts, it is easy to check that we obtain the four types of ball cuts, described in the table of Figure \ref{TableBallCuts}.

The notation \ {\tt b/nb} \ stands for ball/non-ball and \ {\tt G/NG} stands for convexity-gap/non-convexity-gap, respectively.

\begin{figure}[h]
\caption{The six types of cuts. The initial segment $S\in\inii$ is determined by $H=H_S$, where $H$ is the invariance group of the cut. The convex group $H'\supset H$ is defined in equation (\ref{Hprima}).}\label{TableBallCuts}\e

\begin{center}
\as{1.3}
\begin{tabular}{|c|c|c|c|c|}
\hline
$(D,E)$&{\bf $\exists \max(S)$}&$V_f(D)$&$V_i(E)$&{\bf notation} \\\hline
$\left(a+H\right)^+$& yes&$H\;$ \quad stable&$H'\;$ \quad \!stable&$\mbox{\tt (b+G)}^+$ \\\hline
$\left(a+H\right)^+$& no&$H\;$ \quad stable&$H\;$ unstable&$\mbox{\tt (b+NG)}^+$ \\\hline
$\left(a+H\right)^-$& yes&$H'\;$ \quad \!stable&$H\;$ \quad stable&$\mbox{\tt (b+G)}^-$ \\\hline
$\left(a+H\right)^-$& no&$H\;$ unstable&$H\;$ \quad stable&{\tt (b+NG)}$^-$ \\\hline
non-ball& yes&$H'\;$ \quad \!stable&$H'\;$ \quad \!stable&{\tt nb+G} \\\hline
non-ball& no&$H\;$ unstable&$H\;$ unstable&{\tt nb+NG} \\\hline
\end{tabular}
\end{center}
\end{figure}

\noindent{\bf Examples. }Let us exhibit an example of each type of cut. In all examples, the invariance group is $H=0$ and the corresponding initial segment is $S=I$.\e

Take $\g=\Q$, and let $\xi\in\R\setminus\Q$ be an irrational number. 
$$
0^+ \quad \mbox{\tt (b+G)}^+,\qquad 0^-\quad \mbox{\tt (b+G)}^-,\qquad \left(\Q_{<\xi},\Q_{>\xi}\right)\quad \mbox{\tt nb+G}.
$$

Take the Hahn sum $\g=\Q^{(\N)}$, and let $\xi\in\Q^\N\setminus\Q^{(\N)}$ be a vector with an infinite number of nonzero coordinates. 
$$
0^+ \quad \mbox{\tt (b+NG)}^+,\qquad 0^-\quad \mbox{\tt (b+NG)}^-,\qquad
\left(\g_{<\xi},\g_{>\xi}\right)\quad \mbox{\tt nb+NG}.
$$\vskip0.2cm

The shift $(D,E)\mapsto (D,E)+a$ preserves the type of all cuts.
However, multiplication by $-1$ acts on the six types as follows
$$
\mbox{\tt (b+G)}^+\ \leftrightarrow\ \mbox{\tt (b+G)}^-,\qquad 
\mbox{\tt (b+NG)}^+\ \leftrightarrow\ \mbox{\tt (b+NG)}^-,\qquad 
\mbox{\tt nb+G}\,\circlearrowleft,\qquad 
\mbox{\tt nb+NG}\,\circlearrowleft. 
$$

\section{Small extensions of ordered groups}\label{secSmall}



The \emph{rational rank} of an abelian group $G$ is the dimension of its divisible hull as a $\Q$-vector space:
 $$\rrk(G)=\dim_\Q(G\otimes_\Z\Q).$$ 
 
An extension of ordered groups $\g\hk\La$ is \emph{commensurable} if $\rrk(\La/\g)=0$. That is, if $\La/\g$ is a torsion group.\e

The extension $\g\hk\gq$ is simultaneously the minimal divisible extension of $\g$ and the maximal commensurable extension of $\g$.

\begin{lemma}\label{MaxComm}
For any commensurable extension $\g\hk \La$, there exists a unique embedding of $\La$ into $\gq$ such that  the composition $\g\hk \La\hk\gq$ is the canonical embedding.
\end{lemma}

Two extensions $\ \g\hk \La$, $\ \g\hk \La'$ \ 
are said to be \emph{$\g$-equivalent} if there is an isomorphism $\ \La\ism\La'$ \ of ordered groups fitting into a commmutative diagram:
$$
\as{1.2}
\begin{array}{ccc}
\La&&\\
\uparrow&\searrow&\\
\g&\lra&\La'
\end{array}
$$
In this case, we write $\La\sim_\g\La'$.
By Lemma \ref{MaxComm}, every commensurable extension of $\g$ is $\g$-equivalent to a unique subgroup of $\gq$.

\subsection{Small extensions}\label{subsecSmall}
For an arbitrary extension $\iota\colon \g\hk \La$, let $\d\subset \La$ be the relative divisible closure of $\g$ in $\La$:
$$
\d=\left\{x\in \La\mid nx\in\iota(\g),\ \mbox{for some}\ n\in\N\right\}.
$$
Equivalently, $\d$ is the maximal commensurable extension of $\g$ in $\La$.\e

\noindent{\bf Definition.} We say that $\g\hk \La$ is a \emph{small extension} if $\La/\d$ is a cyclic group.\e

Therefore, a small extension is either commensurable ($\d=\La$), or it has $\rrk(\La/\g)=1$ and the quotient  $\La/\d$ is isomorphic to $\Z$.\e

This definition is motivated by the following result.

\begin{theorem}\label{AllSmall}
Let $K$ be a field and let $\mu\colon \kx\to\La\infty$ be a valuation on the polynomial ring $\kx$. Let $\g=\mu(K^*)$ and let $\gm$ be the subgroup of $\La$ generated by $\mu(\kx)$. Then, $\g\subset\gm$ is a small extension of ordered groups.
\end{theorem}

This theorem was proved in  \cite[Thm. 1.5]{Kuhl} for valuations with  trivial support; that is, valuations that may be extended to the rational field $K(x)$. For  valuations with non-trivial support the extension $\g\subset\gm$ is commensurable, because $\gm$ is the value group of an extension of $\mu_{\mid K}$ to a finite extension of $K$.  

Not all small extensions arise from valuations on a polynomial ring. In \cite{Kuhl} it is shown that the divisible closure of $\g$ in $\gm$ must be countably generated, and it must be finitely generated if $\rrk(\gm/\g)=1$.  

Let us exhibit a few examples of small and non-small extensions. Consider the following four extensions of $\g=\Z$:\e

\as{1.3}
\begin{tabular}{ll}
(a)\quad $\Z\subset \Z\oplus\root3\of2\,\Z$,& \qquad(b)\quad $\Z\subset\Z[\root3\of2]$,\\
(c)\quad $\Z\hk \Q\times\Z,\quad m\mapsto(m,0)$,&\qquad(d)\quad $\Z\hk \Q\times\Z,\quad m\mapsto(0,m)$.
\end{tabular}\e
\as{1}

The extensions (a) and (b) preserve the rank, while (c) and (d) increase the rank by one. On the other hand, only (a) and (c) are small.

Actually, all small extensions ``increase the rank at most by one". Let us be more precise about the meaning of this statement.



By Lemma \ref{rkEmbed}, any extension $\g\hk\La$ induces two embeddings of ordered sets
$$
\cv(\g)\hra\cv(\La),\qquad \pcv(\g)\hra\pcv(\La).
$$

The following inequality follows easily from Hahn's embedding theorem:
\begin{equation}\label{rrrk}
\rrk(\La/\g)\ge \sharp \left(\pcv(\La)\setminus\pcv(\g)\right),
\end{equation}
where we identify $\pcv(\g)$ with its image in $\pcv(\La)$ under the above embedding.

Now, it is easy to deduce from Lemmas \ref{ConvI} and \ref{IConv} that
$$\sharp \left(\pcv(\La)\setminus\pcv(\g)\right)=0\sii \sharp \left(\cv(\La)\setminus\cv(\g)\right)=0.$$
$$\sharp \left(\pcv(\La)\setminus\pcv(\g)\right)=1\sii \sharp \left(\cv(\La)\setminus\cv(\g)\right)=1.$$ \vskip0.2cm

\noindent{\bf Definition.} We say that the extension $\ \g\hk\La\ $ \emph{increases the rank at most by one} if $$\sharp \left(\pcv(\La)\setminus\pcv(\g)\right)\le1.$$

If $\sharp \left(\pcv(\La)\setminus\pcv(\g)\right)=0$ we say that $\g\hk\La$ \emph{preserves the rank}.

If $\sharp \left(\pcv(\La)\setminus\pcv(\g)\right)=1$ we say that $\g\hk\La$ \emph{increases the rank by one}.\e

Therefore, the following result follows immediately from (\ref{rrrk}).

\begin{lemma}\label{small<=1}
The extension $\g\hk\gq$ preserves the rank and every small extension $\g\hk\La$ increases the rank at most by one 
\end{lemma}

\noindent{\bf Caution!}  If $\rk(\g)$ is infinite, this terminology abuses of language. 
If $\g\hk\La$ preserves the rank, then obviously $\rk(\g)=\rk(\La)$, but the converse is not true.

For instance, $\N_0=\{0\}+\N$ is isomorphic to $\N$ as an ordered set; hence, the ordered groups $\R^{\N}_{\lx}$ and  $\R^{\N_0}_{\lx}$ have the same rank. However, the natural embedding $\R^{\N}_{\lx}\hk\R^{\N_0}_{\lx}$ increases the rank by one.\e

\subsection{Proper small extension generated by a cut}\label{subsecGD}
In this section, we assume that $\g=\gq$ is a divisible group. 

For any cut $D=(D^L,D^R)$ in $\g$, we consider a formal symbol $x=x_D$ and we build up the abelian group 
$$
\gd= x\Z\oplus\g=\left\{mx+b\mid m\in\Z,\ b\in\g\right\}.
$$ 

There is a unique ordering on $\gd$ which is compatible with the group structure and satisfies
$D^L<x<D^R$.
Namely, 
$$
mx+b\le nx+a\sii (m-n)x\le a-b\sii (m-n)D^L\le a-b.
$$

Therefore, $\g\subset \gd$ is a proper (incommensurable) small extension of ordered abelian groups. 

\begin{lemma}\label{Gshift}
Let $D=(D^L,D^R)$ be a cut in $\g$ and take $a\in \g$. Then, the three extensions $\gd$, $\g(-D)$, $\g(D+a)$ are $\g$-equivalent.
\end{lemma}

\begin{proof}
The $\g$-isomorphism between $\gd$ and $\g(-D)$ is determined by $x_D\mapsto-x_{-D}$.   
The $\g$-isomorphism between $\gd$ and $\g(D+a)$ is determined by $x_D\mapsto x_{D+a}-a$.   
\end{proof}


Our first aim is to characterize those cuts $D$ for which the extension $\g\subset\gd$ preserves the rank. Let us denote  
$$
I=\pcv(\g),\qquad I(D)=\pcv(\gd).
$$


Let $\val_\g$ be the natural valuation on $\g$ and let us denote simply by $\val$ the natural valuation on $\gd$. By Lemma \ref{rkEmbed}, we have a commutative diagram
$$
\as{1.4}
\begin{array}{rcl}
\g&\hra&\gd\\
\mbox{\tiny$\val_\g$}\downarrow &&\ \ \downarrow\mbox{\tiny$\val$}\\
I&\hra&I(D)
\end{array}
$$
and the image of $I$ inside $I(D)$ is $\val(\g)$. Thus, we want to find out for which cuts the equality $I(D)=\val(\g)$  holds.

From now on, we denote simply by $H=H(D)$ the invariance group of our cut $D$. 
Also, we denote by  $\overline{H}$ the convex subgroup of $\gd$ generated by $H$.

The arguments that follow are inspired in \cite[Sec. 3.7]{KuhlCuts}. 

\begin{lemma}\label{cutHD}
The cut in $I(D)$ associated to $\overline{H}$ in Lemma \ref{IConv} is $\left(\val(x+\g),\,\val(H)\right)$. 
\end{lemma}

\begin{proof}For $a\in\g$, suppose that $x-a\in  \overline{H}$. Then, there exists  $h\in H$ such that  $0<|x-a|<h$. If $a<x$, we deduce $x<a+h$, while for $a>x$ we deduce $a-h<x$. Both situations contradict the fact that $D+h=D$. 

Therefore, we have $x-a\not\in \overline{H}$ for all $a\in\g$. This implies two facts:

(1) $\val(H)=\val(\overline{H})$ is the final segment of the cut in $I(D)$ associated to $\overline{H}$.

(2) $\val(x+\g)<\val(\overline{H})$.

Hence, $\val(x+\g)\subset\val(\gd\setminus \overline{H})$ and we need only to prove the opposite inclusion. 

Take any positive $h\in \gd\setminus \overline{H}$. If $h=mx+c$ with $c\in \g$ and $m \ne0$, then 
$$\val(h)=\val(x+(c/m))\in\val(x+\g).$$ 

Now, suppose that $m=0$: that is, $h\in \g\setminus H$. By Lemma \ref{propertiesHD}(i), $h=b-a$ for some $a\in D^L$, $b\in D^R$.
From $a+h=b>x$ we deduce $0<x-a<h$, so that $\val(h)\le\val(x-a)$. If the inequality is strict, then the property (VAL1) of the natural valuation shows that 
$$\val(x-a-h)=\min\{\val(h),\val(x-a)\}=\val(h).
$$
In any case, we see that $\val(h)$ belongs to $\val(x+\g)$.
\end{proof}

\begin{lemma}\label{nonI=max}

Take $a\in\g$. Then, 
$$\val(x-a)\not\in\val(\g)\ \sii\ H(x-a)\cap\g =H\ \sii\ H(x-a)\cap\g \subset H.
$$

In this case, we have $\val(x-a)=\max\left(\val(x+\g)\right)$.
\end{lemma}

\begin{proof}
By Lemma \ref{cutHD}, $\overline{H}\subsetneq H(x-a)$ for all $a\in\g$. Thus, $H\subset H(x-a)\cap\g$.

Suppose that $\val(x-a)\not\in \val(\g)$. 

If $\val(x-a)<\val(x-b)$ for some $b\in\g$, then $v(b-a)=\val(x-a)$ by (VAL1), contradicting our assumption. 
This shows that $\val(x-a)=\max\left(\val(x+\g)\right)$.

Now, for all $h\in\g\setminus H$, Lemma \ref{cutHD} shows that $\val(h)\in\val(x+\g)$. Hence,  $\val(h)<\val(x-a)$ and this implies $h\not\in H(x-a)\cap\g$. Thus, $H(x-a)\cap\g =H$.

On the other hand, if $\val(x-a)=\val(h)\in \val(\g)$, then $h\not\in H$ by Lemma \ref{cutHD}. Therefore, $H(x-a)\cap\g=H(h)\supsetneq H$. 
\end{proof}

The following result follows immediately from Lemmas \ref{cutHD} and \ref{nonI=max}.

\begin{theorem}\label{initSeg}
Let $H=H(D)$ be the invariance group of a cut $D$ in $\g$.\e

(1) \ If $\g\subset\gd$ preserves the rank, then $\val(x+\g)=\val(\g\setminus H)$.\e

(2) \ If $\g\subset\gd$ increases the rank, then, as ordered sets, $$\val(x+\g)=\val(\g\setminus H)+\{\val(x-a)\}$$  where $\val(x-a)$ is the unique principal convex subgroup of $\gd$ which is not generated by an element in $\g$.
In this case,  $H(x-a)\cap \g=H$.
\end{theorem}

\begin{theorem}\label{ballsIncrease}
The extension $\g\subset\gd$ increases the rank if and only if $D$ is a ball cut.
\end{theorem}

\begin{proof}
Suppose that $D=\left(a+H\right)^+$ for some $a\in \g$. Then, $H=H(D)$. 
By Lemma \ref{Gshift}, in order to show that the extension $\g\subset\gd$ increases the rank, we may assume that $a=0$; that is,   
$$D^L=H^+:=\left\{d\in\g\mid d\le h\ \mbox{ for some }h\in H\right\}.$$

Let $x=x_D\in\gd$. By Lemma \ref{nonI=max}, it suffices to show that
$H(x)\cap \g\subset H$ to conclude that $\g\subset\gd$ increases the rank.
Now, a positive $h\in H(x)\cap \g$ satisfies $0< h<nx$ for some $n\in\N$. Hence, $h/n<x$, so that $h/n\in D^L=H^+$. This implies $0<h/n\le h'$ for some $h'\in H$. Thus, $h\in H$. 

By Lemma \ref{Gshift}, for $D=\left(a+H\right)^-=-(-a+H)^+$ the extension $\g\subset\gd$ increases the rank too.

Conversely, suppose that the extension $\g\subset\gd$ increases the rank. Let $x=x_D$ and $H=H(D)$. By Lemma \ref{nonI=max}, there exists $a\in\g$ such that $H(x-a)\cap \g= H$. By Lemma \ref{Gshift}, the $\g$-isomorphism between $\gd$ and $\g(D-a)$ maps $x-a$ to $x_{D-a}$; thus, by replacing $D$ with $D-a$, we may assume that  $H(x)\cap \g= H$. Also, by replacing $D$ with $-D$, if necessary, we may assume that $x>0$. Let us show that  $D=H^+$ under these assumptions.

Since $0\in D^L$ and $H$ is the invariance group of $D$, we have $H\subset D^L$. Thus, $H^+\subset D^L$.  Conversely, take any $d\in D^L$. If $d\le 0$, then $d\in H^+$ because $0\in H$. If $d>0$, then 
$0<d<x$ implies that $d$ belongs to $H(x)\cap \g= H$ too.
\end{proof}

\subsection{Classification of small extensions}\label{subsecClassif}

We keep assuming that $\g$ is a divisible group.
Let $\xx\ppr=\xx\ppr(\g)$ be the set of all pairs $(\La,x)$, where $\La$ is a proper small extension of $\g$ and $x\in\La$ is the choice of a generator of the infinite cyclic group $\La/\g$.

Every pair $(\La,x)\in\xx\ppr$ determines a cut $D_x$ of $\g$, whose initial segment is:
$$
(D_x)^L=\{a\in\g\mid \iota(a)<x\}\subset\g,
$$
where $\iota\colon \g\hk\La$ is the embedding of $\g$ into $\La$. This determines a pre-ordering on $\xx\ppr$, by defining $(\La,x)\le(\La',x')$ whenever $D_x\le D_{x'}$.  

We establish on $\xx\ppr$ the equivalence relation 
$$
(\La,x)\simeq(\La',x')
$$
if there exists an isomorphism of ordered groups between $\La$ and $\La'$ which maps $x$ to $x'$ and acts as the identity on $\g$.

The pre-ordering on $\xx\ppr$ induces a total ordering on the quotient set $\xx\ppr/\!\simeq$.

\begin{proposition}\label{X=cuts}
For a divisible group $\g$, any mapping
$$
\cuts(\g)\lra \xx\ppr,\qquad D\longmapsto \left(\gd,x_D\right)
$$
induces an isomorphism of ordered sets between $\cuts(\g)$ and  $\xx\ppr/\!\simeq$.
\end{proposition}

\begin{proof}This follows from two trivial remarks.\e

(1) $\left(\gd,x_D\right)\simeq\left(\g(D'),x_{D'}\right)\sii D=D'$.\e

(2) For any pair $(\La,x)\in \xx\ppr$, we have $(\La,x)\simeq \left(\g(D_x),x_{D_x}\right)$.
\end{proof}

For any pair $\ep\in\{\pm1\}$, $a\in\g$, consider the bijective mapping
$$
\tea\colon \cuts(\g)\lra \cuts(\g),\qquad D\longmapsto \tea(D)=\ep D+a
$$

These mappings form a subgroup : 
$$\gpm:=\left\{\tea\mid\ep\in\{\pm1\},\ a\in\g\right\}$$ 
of the group of all bijective mappings from $\cuts(\g)$ onto itself, with the operation of composition. Since,
$$
\tea\circ t_{\ep',a'}=t_{\ep\ep',\ep a'+a},
$$
the group $\gpm$ is a semidirect product of $\g$ by $\{\pm 1\}$.

\begin{lemma}\label{groupAction}
Let $D,D'$ be two cuts in $\g$. The small extensions $\gd$, $\g(D')$ are $\g$-equivalent if and only if  $D'=\tea(D)$ for some  $\tea\in \gpm$.   
\end{lemma}

\begin{proof}
If $D'=\tea(D)$ for some  $\tea\in \gpm$, then $\gd\sim_\g\g(D')$ by Lemma \ref{Gshift}.

Conversely, if $\varphi\colon \g(D')\iso\gd$ is a $\g$-isomorphism, then 
$$
\gen{\g,x_D}=\gd=\gen{\g,\varphi(x_{D'})},
$$
where $\gen{\g,x_D}$ is the subgroup generated by $\g$ and $x_D$.

Hence, there exist $\ep\in\{\pm1\}$, $a\in\g$ such that $\varphi(x_{D'})=\ep\,x_D+a$.
Since 
$$\left(\g(D'),x_{D'}\right)\simeq\left(\gd,\varphi(x_{D'})\right)=\left(\gd,\ep\,x_D+a\right)\simeq\left(\g(\ep D+a),x_{\ep D+a}\right),$$
Proposition \ref{X=cuts} shows that $D'=\ep D+a=\tea(D)$.
\end{proof}

Therefore, the $\g$-equivalence classes of small extensions of $\g$ are parametrized by
$$
\{\g\}\sqcup \left(\cuts(\g)/\g^\pm\right).
$$

\subsection*{The general case} For an arbitrary ordered abelian group $\g$, the set $\Sme(\g)$ of $\g$-equiva\-lence classes of small extensions of $\g$ may be stratified as
$$
\Sme(\g)=\bigsqcup_{\g\subset\d\subset\gq}\Sme_\d(\g),
$$
where $\d$ runs through all subgroups of $\gq$ containing $\g$, and $\Sme_\d(\g)$
is  the set of the $\g$-equivalence classes of  
small extensions for which the relative divisible closure of $\g$ is $\g$-isomorphic to $\d$. 

Let us show that $\Sme_\d(\g)$ admits a complete analogous description as $\Sme(\gq)$:
\begin{equation}\label{smeDelta}
\Sme_\d(\g)=\{\d\}\sqcup \left(\cuts(\gq)/\d^\pm\right).
\end{equation}

Indeed, let $\xx(\d)\ppr\subset\xx(\gq)\ppr$ be the subset of pairs $(\La,x)$ where $\La$ is a small extension of $\g$ such that the divisible closure of $\g$ in $\La$ is $\g$-isomorphic to $\d$. The equivalence relation $\simeq$ ``descends" to $\xx(\d)\ppr$; that is, if $(\La,x)\in\xx(\d)\ppr$, then the whole class of $(\La,x)$ is included in $\xx(\d)\ppr$.\e

\begin{proposition}\label{Xd=cuts}
For all $D\in\cuts(\gq)$, let $\d(D)$ be the subgroup of $\gq(D)$ ge\-ne\-rated by $\d$ and $x_D$.
Then, any  mapping$$
\cuts(\gq)\lra \xx(\d)\ppr,\qquad D\longmapsto \left(\d(D),x_D\right)
$$
induces an isomorphism of ordered sets between $\cuts(\gq)$ and  $\xx(\d)\ppr/\!\simeq$.
\end{proposition}

\begin{proof}
 
This follows from two remarks.\e

(1) $\left(\d(D),x_D\right)\simeq\left(\d(D'),x_{D'}\right)\sii D=D'$.\e

(2) For all $(\La,x)\in \xx(\d)\ppr$, we have $(\La,x)\simeq \left(\d(D),x_{D}\right)$ for some $D\in\op{Cuts}(\gq)$.\e

Let us denote $x=x_D$, $x'=x_{D'}$. We have $\d(D)=x\Z\oplus\d$, $\d(D')=x'\Z\oplus\d$. Suppose there exists an order preserving isomorphism
$j\colon \d(D)\iso\d(D')$ acting as the identity on $\g$, and such that $j(x)=x'$. Since $x$ and $x'$ have no torsion over $\d$, we have necessarily $j(\d)=\d$. Since $\d/\g$ is a torsion group, the only ordered $\g$-automorphism of $\d$ is the identity. Hence, $j$ acts as the identity on $\d$. Now, for all $a\in\gq$, there exists $n\in \N$ such that $na\in\d$, so that
$$
a<x\ \sii\ na<nx\ \sii\ na=j(na)<j(nx)=nx'\ \sii\ a<x'.
$$
Therefore, $D=D'$. This proves (1).

Let us prove (2). For all $(\La,x)\in \xx(\d)\ppr$, there is an isomorphism $\La\simeq x\Z\oplus\d$ which determines an ordering on the latter group. Now, the group $x\Z\oplus\gq$ admits a unique ordering for which the natural inclusion 
$x\Z\oplus\d\subset x\Z\oplus\gq$ is an embedding of ordered groups. Indeed, for any $a,b\in\gq$ and integers $m>m'$, we have:
$$
mx+a< m'x+b\sii x<(b-a)/(m-m')\sii nx<n(b-a)/(m-m'), 
$$
where $n\in\N$ satisfies $n(b-a)/(m-m')\in \d$.
Thus, $x$ determines a cut $D_x$ in $\gq$ and clearly $(\La,x)\simeq \left(\d(D_x),x_{D_x}\right)$.
\end{proof}

A result completely analogous to Lemma \ref{groupAction}, shows that the $\g$-equivalence classes of incommensurable small extensions such that the divisible closure of $\g$ is isomorphic to $\d$ are parametrized by the set $\,\cuts(\gq)/\d^\pm$. This proves (\ref{smeDelta}).

\section{Embedding small extensions in real vector spaces}\label{secRealVS}

By using Hahn's embedding theorem, our ordered abelian group $\g$ may be embedded in a suitable real vector space \cite[Sec. A]{Rib}. In this section, we recall this construction and we find a concrete real vector space which is sufficiently large to contain as well all small extensions of $\g$, up to $\g$-equivalence.

We give two applications of these real models for small extensions. 
On one hand, these models contain realizations of all cuts in the divisible hull $\gq$ of $\g$ in a real vector space. These realizations facilitate the computation of the cofinality and coinitiality of all these cuts (Sections \ref{subsubCofNB} and \ref{subsubCofB}).

On the other hand, in Section  \ref{secVals}, we shall derive another application of these real models to the problem  of classifying valuations on polynomial rings, up to equivalence.

\subsection{Embedding rank-preserving small extensions in real vector spaces}\label{subsecSameRank}

For our ordered group $\g$ with skeleton $\left(I;(C_i)_{i\in I}\right)$, the skeleton of $\gq$ is
$$
\left(I;(Q_i)_{i\in I}\right),\qquad Q_i=C_i\otimes_\Z\Q\ \mbox{ for all }i\in I.
$$

By Lemma \ref{small<=1}, we have a natural identification:
$$I=\pcv(\g)=\pcv(\gq).$$

By Hahn's Theorem  \ref{hahn}, there is a (non-canonical) immediate $\Q$-linear embedding 
$$
\gq\hooklongrightarrow\hq:=\lower.5ex\hbox{\mbox{\Large$\mathbf{H}$}}{\lower1.3ex\hbox{\mbox{\tiny${i\!\!\in\!\! I}$}}}\,Q_i.
$$


For each $i\in I$ we fix, once and for all, a positive element $1^i\in Q_i$. 
As shown in Proposition  \ref{rk1}, this choice determines an embedding $Q_i\hk\R$ of ordered groups, which sends our fixed element $1^i$ to the real number $1$. 
In this way, we get an embedding $\hq \hk \rlex$ which obviously preserves the rank.

Altogether, we obtain a rank-preserving extension
$$
\g\ \hra\ \gq\ \hra\ \hq\ \hra\ \rlex,
$$
which is maximal among all rank-preserving extensions of $\g$ \cite[Sec. A]{Rib}.

\begin{theorem}\label{MaxEqRk}
For any rank-preserving extension $\g\hk \La$, there exists an embedding $\La\hk\rlex$ fitting into a commutative diagram
$$
\as{1.2}
\begin{array}{ccc}
\La&&\\
\uparrow&\searrow&\\
\g&\lra&\rlex
\end{array}
$$
\end{theorem}

\noindent{\bf Caution!} The embedding $\La\hk\rlex$ is not unique. Every rank-preserving extension of $\g$ is $\g$-equivalent to some subgroup of $\rlex$, but not to a unique one.  \e

\emph{From now on, we identify $\g$, $\gq$ and $\hq$ with their image in $\rlex$. }\e

As we saw in Section \ref{subsecImmed}, the convex subgroups of $\rlex$ are given by:
$$
H_S=\left\{x=(x_i)_{i\in I}\in\rlex\mid x_i=0\ \mbox{for all }i\in S\right\},\qquad S\in\inii.
$$
The convex subgroups of $\gq$ are obtained as $H_S\cap\gq$ for $S\in\inii$.

For all $S\in\inii$, consider the truncation by $S$:
$$
\pi_S\colon \rlex\lra\rlex,\qquad x=(x_i)_{i\in I}\longmapsto x_S=(y_i)_{i\in I},
$$
where $y_i=x_i$ for all $i\in S$ and $y_i=0$ otherwise. Note that $\pi_S^{-1}(x_S)=x+H_S$.

By Lemma \ref{hahnSpecial}, $\gq$ contains the Hahn sum $\bigoplus_{i\in I}Q_i$ and we have
\begin{equation}\label{gqgq}
a\in\gq\ \imp\ a_S\in \gq, \quad\mbox{\,for all }S\in\inii. 
\end{equation}

\noindent{\bf Notation. }For all $i\in I$, consider unit vectors $e_i=(x_j)_{j\in I}\in\rlex$, with zero coordinates everywhere except for $x_i=1$. For an arbitrary real number $\xi\in \R$, the vector $\xi e_i\in \rlex$ has an obvious meaning. Clearly, $\xi e_i\in\gq$ if and only if $\xi\in Q_i\subset\R$.


\subsubsection{A real model for the set of non-ball cuts in $\gq$}

All $x\in\rlex\setminus\gq$ determine a cut $D_x\in\cuts(\gq)$, with initial and final segments:
$$
(D_x)^L=\left\{a\in\gq\mid a<x\right\},\qquad
(D_x)^R=\left\{a\in\gq\mid a>x\right\}.
$$
The subgroup $\gen{\gq,x}\subset\rlex$, generated by $\gq$ and $x$, is obviously $\g$-equivalent to the group $\gq(D_x)$ constructed in Section \ref{subsecGD}.
Since the extension $\g\subset \gen{\gq,x}\subset\rlex$ preserves the rank,  
Theorem \ref{ballsIncrease} shows that all these $D_x$ are non-ball cuts.

Conversely, for each non-ball cut $D\in\cuts(\gq)$, Theorems \ref{ballsIncrease} and \ref{MaxEqRk} show that  the incommensurable small extension $\gq(D)$ is $\g$-equivalent to some subgroup of $\rlex$.  In particular, $D$ is realized by some elements in $\rlex\setminus \gq$. 

The aim of this section is to find a subset in $\rlex$ which serves as a  model for the set of non-ball cuts in $\gq$. More precisely, we shall construct a canonical subset $\gnbc\subset\rlex\setminus\gq$ such that the following mapping is an isomorphism of ordered sets:
$$
\gnbc\lra\left\{\mbox{non-ball cuts in }\gq \right\},\qquad x\longmapsto D_x.
$$

In particular, the set $\gnbc$ must be a set of representatives of the following equivalence relation on $\rlex\setminus \gq$:
$$
x\sim_\cut x'\ \sii\ D_x=D_{x'}.
$$
Let us denote by $\cutcl{x}$ the class of $x$. 

\begin{lemma}\label{class=HS}
For $x\in\rlex\setminus\gq$, let $H\subset \rlex$ be the convex subgroup generated by the invariance group of the cut $D_x$ in $\gq$. Then, $\cutcl{x}=x+H\subset \rlex\setminus\gq$.
\end{lemma}

\begin{proof}
For all $h_0\in H\cap\gq$, since $D_x+h_0=D_x$, we deduce $D_x=D_{x+h_0}$. On the other hand, all positive $h\in H$ satisfy $0< h \le h_0$ for some $h_0\in H\cap \gq$. Hence,
$$
(D_x)^L<x<x+h\le x+h_0<(D_x)^R,\qquad (D_x)^L<x-h_0\le x-h<x<(D_x)^R.
$$
Therefore, for all $h\in H$ we have $x+h\not\in \gq$ and $D_{x+h}=D_x$. Thus, $x+H\subset\cutcl{x}$.

To prove the converse inclusion, suppose that $y\in\cutcl{x}$; or equivalently, $D_x=D_{y}$. Let $S\in\inii$ be determined by $H=H_S$. We want to see that $y\in x+H_S$; that is, $x_S=y_S$. 

Suppose that $x_S<y_S$ and let $i\in S$ be the minimal index for which $x_i<y_i$. Take any positive $q\in Q_i$ such that $x_i<x_i+q<y_i$. Consider the element $b=q e_i\in\gq$. Since $q\ne0$ and $i\in S$, we have $b\not \in H$; hence, there exists $a\in (D_x)^L$ such that $a+b\not \in (D_x)^L=(D_y)^L$. In other words, $a+b>y$. However, from $a<x$ we deduce, by the construction of $b$, that $a+b<x+b<y$. This is a contradiction.  
\end{proof}

Let us compute the invariance group of the cut $D_x$ for all $x\in\rlex\setminus\gq$.

\begin{lemma}\label{MinExists}
For all $x\in\rlex\setminus\gq$ consider the set $\ss(x)=\left\{T\in\inii\mid x_T \not\in\gq\right\}$.
Then, the invariance group of $D_x$ is $H_S\cap \gq$, where $S=\min(\ss(x))$. 
\end{lemma}

\begin{proof}
Let  $S\in\inii$ be determined by $H(D_x)=H_S\cap\gq$. By Lemma \ref{class=HS},  $x_S\not\in\gq$, so that $S\in \ss(x)$. Let us show that $S$ is the minimal element in this set. 

Let $T\in\inii$ such that $x_T\not\in\gq$. Then, $D_x=D_y$ for all $y\in x+H_T$. Indeed, for all $a\in\gq$,  the equality $a_T=x_T$ is impossible by (\ref{gqgq}). Hence,
$$
a<y\sii a_T<y_T=x_T\sii a<x.
$$
Thus, $y\sim_\cut x$ and Lemma \ref{class=HS} shows that $x+H_T\subset x+H_S$. This implies $H_T\subset H_S$, or equivalently, $T\supset S$.
\end{proof}

The set of cuts admits a stratification by the invariance group of the cut. Take any $S\in\inii$. By Lemmas \ref{class=HS} and \ref{MinExists}, the non-ball cuts of $\gq$ whose invariance group is $H_S\cap\gq$ are parametrized by the following subset of $\rlex$:
$$
\op{nbc}(S)=\left\{x\in\rlex\setminus \gq\mid \supp(x)\subset S\ \mbox{ and }\ x_T\in\gq \mbox{ for all } T\in\inii,\ T\subsetneq S \right\}.
$$
Hence, the real model we are looking for is:
$$
\gnbc:=\bigsqcup_{S\in\inii}\op{nbc}(S).
$$

By Lemma \ref{convGap}, the two types of non-ball cuts are distinguished by the parameter $S$ as follows:
$$
\mbox{\tt nb+NG}=\bigsqcup_{\nexists\max(S)}\op{nbc}(S),\qquad
\mbox{\tt nb+G}=\bigsqcup_{\exists\max(S)}\op{nbc}(S)
$$

This description of the non-ball cuts yields another way to distinguish the two types of cuts.

\begin{lemma}\label{Srat}
An $x\in\op{nbc}(S)$ is of type {\tt \,nb+NG\,} if and only if $x\in\hq$.
\end{lemma}

\begin{proof}
Suppose that  $x=x_S=(x_j)_{j\in I}\not\in\hq$, and let $J=\supp(x)$. By assumption, $$J^0:=\{j\in J\mid x_j\not\in Q_j\}\ne\emptyset.$$
Since $J$ is well-ordered, there exists $i=\min\left(J^0\right)$. Since $x_i\ne0$, we have $i\in S$. 

Let  $R=I_{\le i}\subset S$. 
Since $x_R\not\in\gq$, 
we must have $R=S$ by the minimality of $S$. Hence, $i=\max(S)$.

Conversely, suppose that $S$ contains a maximal element $i$, and let $T=I_{<i}\subsetneq S$. By  the minimality of $S$, $x_{T}\in\gq$, so that $x_j\in Q_j$ for all $j<i$. This implies that $x_{i}\not\in Q_{i}$, so that $x\not\in\hq$. 

Indeed, if $x_i\in Q_i$, then the element $b=x_ie_i$ belongs to $\gq$ and we get a contradiction: $x=x_T+b\in\gq$. 
\end{proof}

In particular, the set {\tt nb+G} may be described in a more explicit form:
\begin{equation}\label{nbVG}
\mbox{\tt nb+G}=\bigsqcup_{i\in I}\left\{x+\xi e_i\in\rlex\,\mid\, x\in\gq,\ \supp(x)\subset I_{<i},\ \xi\in \R\setminus Q_i\right\}.
\end{equation}

\subsection*{Examples}\mbox{\null}\e


(Ex1) \ $\rk(\g)=1$, \,\quad {\tt nb+NG}\,$=\emptyset$,\quad {\tt nb+G}\,$=\R\setminus\gq$.\e

(Ex2) \ $\g=\R^2_{\lx}$,\qquad {\tt nb+NG}\,$=${\tt nb+G}\,$=\emptyset$.\e

(Ex3) \ $\g=\Q^2_{\lx}$,\qquad {\tt nb+NG}\,$=\emptyset$,\quad {\tt nb+G}\,$=\left((\R\setminus\Q)\times\{0\}\right)\,\sqcup\,\left(\Q\times(\R\setminus\Q)\right)$.\e


(Ex4) \ $\g=\R^{(\N)}$,\qquad {\tt nb+NG}\,$=\R^{\N}\setminus \R^{(\N)}$,\quad {\tt nb+G}\,$=\emptyset$.\e

(Ex5) \ $\g=\Q^{(\N)}$,\qquad {\tt nb+NG}\,$=\Q^{\N}\setminus \Q^{(\N)}$,\quad {\tt nb+G}\,$=\bigsqcup\nolimits_{i\in\N}\left(\Q^{i-1}\times \left(\R\setminus\Q\right)\times \{0\}^{\N_{>i}}\right)$.



\subsubsection{Cofinality and coinitiality of non-ball cuts}\label{subsubCofNB}

For all proper cuts $D$ in $\gq$, we clearly have $\ka(D)=\la(-D)$ and $\la(D)=\ka(-D)$. 




\begin{theorem}\label{CofNB}
Let $D$ be a non-ball cut in $\gq$. Let $S\in \inii$ be the initial segment such that $H_S$ is the invariance group of $D$. Then, the cofinality and coinitiality $(\ka,\la)$ of $D$  take the following values:
\begin{enumerate}
\item If $D$ is an {\tt \,nb+G} cut, then $\ka=\la=\aleph_0$.
\item If $D$ is an {\tt \,nb+NG} cut, then $\ka=\la=\cofin(S)$.
\end{enumerate}
\end{theorem}

\begin{proof}
The two types of non-ball cuts are stable under multiplication by $-1$. Hence, if we check that $\ka(D)$ depends only on the invariance group $H(D)=H(-D)$, we deduce that $\la(D)=\ka(-D)=\ka(D)$.

Suppose that $D$ is an {\tt \,nb+G\,} cut, and let $i_M=\max(S)$. Let $T=I_{<i_M}\subsetneq S$. By (\ref{nbVG}),  $D$ is realised in the real model $\gnbc\subset\rlex\setminus\gq$ by a unique $x\in\gnbc$ of the form
$$
x=x_T+\xi e_{i_M},\qquad x_T\in\gq,\qquad \xi\in\R\setminus Q_{i_M}.  
$$   
Consider a countable sequence $(q_m)_{m\in\N}$ of rational numbers which is cofinal in $\R_{<\xi}$. Then, the countable sequence $\left(x_T+q_me_{i_M}\right)_{m\in \N}$ in $\gq$ is  cofinal in $D^L$. Thus, $\ka=\aleph_0$.
This ends the proof of (1).

Now, suppose that $D$ is an {\tt \,nb+NG\,} cut. By Lemma \ref{Srat}, $D$ is realised in the real model $\gnbc\subset\rlex\setminus\gq$ by a unique $x=(x_i)_{i\in I}\in\gnbc$ of the form
$$
x=x_S\in\hq\setminus\gq,\quad\mbox{such that}\quad x_T\in\gq\quad \mbox{for all }T\subsetneq S.  
$$   

Since $S$ does not contain a maximal element, the subset $\supp(x)\subset S$ is cofinal. Indeed, for any $j\in S$, we have $T=S_{\le j}\subsetneq S$, so that $x_T\in\gq$. Therefore, $x\ne x_T$ and there must exist $i\in\supp(x)$ such that $i>j$. As a consequence, $$\cofin(S)=\cofin(\supp(x)).$$

For all $j\in\supp(x)$, let $T_j=S_{\le j}\subsetneq S$. Let us construct an element $z(j)\in \hq$ as follows:
$$
z(j)=x_{T_j}+q_{j+1}e_{j+1},\qquad q_{j+1}\in Q_{j+1},\ q_{j+1}<x_{j+1},  
$$
where $j+1$ is the immediate successor of $j$ in $\supp(x)$, which is a well-ordered set, by definition.
Clearly, $z(j)\in\gq$ because $x_{T_j}$ and $q_{j+1}e_{j+1}$ belong both to $\gq$. 
Also, $$z(j)<x \quad\mbox{for all }j\in \supp(x).$$ Indeed, for $R=S_{<j+1}$ we have $z(j)_R=x_R$, because all eventual indices $i\in I$ such that $j<i<j+1$ do not belong to $\supp(x)$, so that $x_i=z(j)_i=0$.  We claim that the mapping
$$
\supp(x)\lra D_x^L,\qquad j\longmapsto z(j)
$$
is an order-preserving embedding with cofinal image. This will end the proof of the theorem: \,$\ka=\cofin(D_x^L)=\cofin(\supp(x))=\cofin(S)$.

Let us prove that the mapping $j\mapsto z(j)$ preserves the ordering. If $j<k$ in $\supp(x)$, then for $T=T_{
j+1}=S_{\le j+1}$, we have $\left(z(j)\right)_T<x_T=\left(z(k)\right)_T$. Hence, $z(j)<z(k)$.

Finally, for all $a\in D_x^L$, let $i\in I$ be the minimal index with $a_i<x_i$. We must have $i\in S$, because $i>S$ would imply  $x_S=a_S\in \gq$, which is false. Since $\supp(x)$ is cofinal in $S$, there exists $j\in\supp(x)$ such that $i<j$. Hence, $a<z(j)$, because $a_{T_j}<x_{T_j}=(z(j))_{T_j}$. Thus, the image of  $\supp(x)\to D_x^L$ is cofinal in $D_x^L$.    
\end{proof}



\subsection{Embedding rank-increasing small extensions in real vector spaces}\label{subsecIncRank}

An embedding $\iota\colon I\hk J$ of ordered sets  \emph{adds one element} if  $J\setminus \iota(I)$ is a one-element subset of $J$.

For all $S\in\inii$, consider the ordered set 
\begin{equation}\label{IS}
I_S=S+\{i_S\}+S^c,
\end{equation}
where $i_S$ is a formal symbol. The natural embedding $I\hk I_S$ adds one element. 

Consider the \emph{one-added-element hull} of $I$: $$\I:=I\cup\left\{i_S\mid S\in\inii\right\}.$$
We may consider a natural total ordering on $\I$ determined by

\begin{enumerate}
\item[(i)] For all $S\in\inii$, the restriction of the ordering to $I_S=I\cup\{i_S\}$ is the ordering considered in (\ref{IS}). 
\item[(ii)] $i_S<i_T\ \sii\ S\subsetneq T,\quad $for all $S,T\in\inii$.
\end{enumerate}
Note that $i_\emptyset=\min(\I)$, $i_I=max(\I)$.

\begin{lemma}\label{OAEUniv}
For any  embedding $I\hk J$ of ordered sets that adds one element, there exists a unique embedding $J\hk \I$ fitting into a commutative diagram
$$
\as{1.2}
\begin{array}{ccc}
J&&\\
\uparrow&\searrow&\\
I&\hra&\I
\end{array}
$$
The image of $J$ in $\I$ is $I_S$ for a unique $S\in\inii$.
\end{lemma}

\begin{lemma}\label{SemiUniv}
If the extension  $\g\hk\La$ of ordered groups increases the rank by one, there is a unique $S\in\inii$ and an embedding $\La\hk\R^{I_S}_{\lx}$ fitting into  commutative diagram:
$$
\as{1.4}
\begin{array}{ccccccc}
&&\La&&&&\\
&\nearrow&&\searrow&&&\\
\g&\hra&\rlex&\hra&\R^{I_S}_{\lx}&\hra&\rii
\end{array} 
$$
\end{lemma}

\begin{proof}
The initial segment $S$ is uniquely determined by the condition $\pcv(\La)\simeq I_S$.

The proof  follows easily from Hahn's embedding theorem and Lemma \ref{OAEUniv}.
\end{proof}

Therefore, we are interested in the following subset of $\rii$, containing $\rlex$:
$$\rll=\rll(I):= \bigcup\nolimits_{S\in\inii}\R^{I_S}_{\lx}\,\subset\, \rii.$$

\noindent{\bf Notation. }For each $S\in\inii$ consider unit vectors
$$
e_S=(x_j)_{j\in \I}\in\R^{I_S}_{\lx}\subset \rii,\qquad x_j=0\ \mbox{ for all }j\ne i_S,\quad x_{i_S}=1.
$$

\subsubsection{A real model for the set of ball cuts in $\gq$}

All $x\in\rll\setminus\rlex$ determine a cut $D_x\in\cuts(\gq)$, with initial and final segments:
$$
(D_x)^L=\left\{a\in\gq\mid a<x\right\},\qquad
(D_x)^R=\left\{a\in\gq\mid a>x\right\}.
$$
The subgroup $\gen{\gq,x}\subset\rll$ generated by $\gq$ and $x$ is obviously $\g$-equivalent to the group $\gq(D_x)$ constructed in Section \ref{subsecGD}.
Since the extension $\g\subset \gen{\gq,x}$ increases the rank,  
Theorem \ref{ballsIncrease} shows that all these $D_x$ are ball cuts.

Conversely, for each ball cut $D\in\cuts(\gq)$, Theorem \ref{ballsIncrease} and Lemma \ref{SemiUniv} show that  the incommensurable small extension $\gq(D)$ is $\g$-equivalent to some subgroup of $\rii$, contained in $\rll\setminus\rlex$. Thus, $D$ is realized by some elements in $\rll\setminus\rlex$. 

The aim of this section is to find a canonical subset $\gbc\subset\rll\setminus\rlex$ such that the following mapping is an isomorphism of ordered sets:
$$
\gbc\lra\left\{\mbox{ball cuts in }\gq \right\},\qquad x\longmapsto D_x.
$$
The set $\gbc$ must be a set of representatives of the equivalence relation on $\rll\setminus \rlex$:
$$
x\sim_\cut x'\ \sii\ D_x=D_{x'}.
$$

For all $b\in\gq$, $S\in\inii$, let us denote
$$
b_S^+=b_S+e_S=((b_j)_{j\in S}\mid1\mid0\cdots0)\in\R^{I_S}_{\lx},$$$$ \ \,\,b_S^-=b_S-e_S=((b_j)_{j\in S}\mid-1\mid0\cdots0)\in\R^{I_S}_{\lx},
$$
where $\pm1$ is placed at the $i_S$-th coordinate.

Then, the reader may easily check that we may consider:
$$
\gbc:=\bigsqcup_{S\in\inii}\op{bc}^-(S)\sqcup\op{bc}^+(S),\qquad \op{bc}^\pm(S):=\left\{b_S^\pm\mid b\in\gq\right\}.
$$

Clearly, the corresponding ball cuts in $\gq$ are $\left(b+H_S\right)^\pm$. 
Note that $$b_S^-<b+H_S<b_S^+.$$


\subsubsection{Cofinality and coinitiality of ball cuts}\label{subsubCofB}

Let $D$ be a {\bf proper} ball cut in $\gq$. The invariance group $H=H(D)$ is a proper convex subgroup of $\gq$. Let us denote by $S=(S,S^c)$ the cut in $I$ intrinsically associated to $H$ in Lemma \ref{IConv}, uniquely determined by the condition $H=H_S$. 

In order to compute the cardinal numbers $\ka(D)$, $\la(D)$, let us split the proper ball cuts into eight types, encoded by a sequence of three signs: $\quad (\pm,\pm,\pm)$.\e

The $+/-$ in the 1st coordinate indicates if $D=(a+H)^+$ or $D=(a+H)^-$.\e

The $+/-$ in the 2nd coordinate indicates the existence/non-existence of $\max(S)$.\e

The $+/-$ in the 3rd coordinate indicates the existence/non-existence of $\min(S^c)$.\e

In other words, each one of the four types of ball cuts considered in Section \ref{subsec6types} splits into two subtypes according to the existence or not of $\min(S^c)$:
$$
\mbox{\tt (b+G)}^+=(+,+,\pm),\quad
\mbox{\tt (b+nG)}^+=(+,-,\pm),$$$$\mbox{\tt (b+G)}^-=(-,+,\pm),\quad
\mbox{\tt (b+nG)}^-=(-,-,\pm).
$$
Since $H(D)=H(-D)$, multiplication by $-1$ acts on these eight subtypes by changing the sign in the first coordinate.

Let us denote by $(\ka(S),\la(S))$ the cofinality and coinitiality of the cut $(S,S^c)$ in $I$. 
Since $H$ is a proper subgroup, $S\subset I$ is a non-empty subset. However, for $H=0$ we have $S=I$ and  $S^c=\emptyset$. \e

\noindent{\bf Convention. }For the improper cut $S=(I,\emptyset)$, we agree that $\la(S)=1$.\e

\begin{theorem}\label{CofB}
Let $D$ be a proper ball cut in $\gq$. Let $S\in \inii$ be the initial segment such that $H_S$ is the invariance group of $D$. 
Then, the cofinality and coinitiality of $D$ take the values indicated in  the table in Figure \ref{Table8BallCuts}.
\end{theorem}

\begin{proof}
Let $D=\left(a+H_S\right)^-$, or $D=\left(a+H_S\right)^+$, for some $a\in\gq$. Denote $\ka=\ka(D)$, $\la=\la(D)$. Since the operation of shifting does not change the pair $(\ka,\la)$, we may assume that $a=0$.

In the real model of $\gbc$, the cut $D=H_S^-$ is realised by $x=-e_S=(0\mid -1\mid 0)$, where the $-1$ is placed at the $i_S$-th coordinate.   

If there exists $i_M=\max(S)$,  the countable sequence $\left((-1/n)e_{i_M}\right)_{n\in\N}$ is cofinal in $D_x^L$. Thus, for the cuts of the type $\mbox{\tt \,(b+G)}^-$ we have $\ka=\aleph_0$.

If $S$ has no maximal element, the mapping
$$
S\lra D_x^L,\qquad i\longmapsto -e_i
$$
is an order-preserving embedding with cofinal image. Thus, for the cuts of the type \mbox{\tt \,(b+NG)$^-$\,} we have $\ka=\cofin(S)=\ka(S)$.

In the real model of $\gbc$, the cut $D=H_S^+$ is realised by $x=e_S=(0\mid 1\mid 0)$, where the $1$ is placed at the $i_S$-th coordinate.   

If there exists $i_m=\min(S^c)$,  the countable sequence $\left(ne_{i_m}\right)_{n\in\N}$ is cofinal in $D_x^L$. Thus, $\ka=\aleph_0$.
Now, suppose that $S^c$ has no minimal element. If $S^c=\emptyset$, then $D=0^+$ and $\max(D^L)=0$, so that $\ka(D)=1$. Thus, $\ka=\la(S)$, after our convention. If $S^c\ne\emptyset$, the mapping
$$
\left(S^c\right)\opp\lra D_x^L,\qquad i\longmapsto e_i
$$
is an order-preserving embedding with cofinal image. Thus, $\ka=\coini(S^c)=\la(S)$.

Therefore, for the cuts of the types $\mbox{\tt \,(b+G)}^+$ and \mbox{\tt \,(b+NG)$^+$\,} we have $\ka=\aleph_0$ or $\ka=\la(S)$, according to the existence or not of a minimal element in $S^c$.

Since multiplication by $-1$ changes the sign of the first coordinate, the knowledge of $\ka$ for all types of ball cuts determines the values of $\la=\ka(-D)$ as well.  
\end{proof}

\begin{figure}
\caption{The eight subtypes of ball cuts}\label{Table8BallCuts}\e

\begin{center}
\as{1.1}
\begin{tabular}{|c|c|c|c|}
\hline
type&subtype&$\ka(D)$&$\la(D)$ \\\hline
\multirow{2}{3.5em}{$\mbox{\tt (b+G)}^+$}&$(+,+,+)$&$\aleph_0$&$\aleph_0$ \\
&$(+,+,-)$&$\la(S)$&$\aleph_0$\\\hline
\multirow{2}{4em}{$\mbox{\tt (b+nG)}^+$}&$(+,-,+)$&$\aleph_0$&$\ka(S)$ \\
&$(+,-,-)$&$\la(S)$&$\ka(S)$ \\\hline
\multirow{2}{4em}{$\mbox{\tt (b+G)}^-$}&$(-,+,+)$&$\aleph_0$&$\aleph_0$ \\
&$(-,+,-)$&$\aleph_0$&$\la(S)$ \\\hline
\multirow{2}{4em}{$\mbox{\tt (b+nG)}^-$}&$(-,-,+)$&$\ka(S)$&$\aleph_0$ \\
&$(-,-,-)$&$\ka(S)$&$\la(S)$ \\\hline
\end{tabular}
\end{center}
\end{figure}

\section{Applications to valuation theory}\label{secVals}

Consider two valuations on a polynomial ring $\kx$ over a field $K$:
$$\mu\colon \kx\to\Omega\infty,\qquad \nu\colon \kx\to\Omega'\infty.$$
Let $\p_\mu=\mu^{-1}(\infty)$, $\p_\nu=\nu^{-1}(\infty)$ be their supports.
Let $\gm\subset\Omega$, $\gn\subset\Omega'$ be the subgroups generated by $\mu(\kx\setminus\p_\mu)$, $\nu(\kx\setminus\p_\nu)$, respectively.

We say that $\mu$ and $\nu$ are \emph{equivalent}, and we write $\mu\sim\nu$, if there is an isomorphism of ordered groups $\varphi\colon \gm \ism\gn$ fitting into a commutative diagram
$$
\as{1.4}
\begin{array}{ccc}
\gm\infty&\stackrel{\varphi}\lra\ &\!\!\gn\infty\\
\quad\ \mbox{\scriptsize$\mu$}&\nwarrow\ \nearrow&\!\!\!\mbox{\scriptsize$\nu$}\quad\\
&\kx&
\end{array}
$$

Let us denote by $[\mu]$ the class of all valuations on $\kx$ which are equivalent to $\mu$.\e

Let $(K,v)$ be a valued field, with group of values $\g=v(K^*)$. Let $I=\pcv(\g)$. As we saw in the last section, we may fix an embedding  of ordered groups, 
$$
\ell\colon \g\hra\rlex\subset\rll\subset\rii.
$$

Let us focus on the problem of describing all equivalence classes of valuations on $\kx$ 
whose restriction to $K$ is equivalent to $v$.

\subsection{Equivalence classes of valuations on $\kx$}
If the restriction to $K$ of the valuation $\mu\colon \kx\to\Omega\infty$ is equivalent to $v$, there exists an embedding of ordered groups $\iota\colon\g\hk\Omega$, fitting into a commutative diagram 
\begin{equation}\label{iota}
\as{1.5}
\begin{array}{ccc}
\kx&\stackrel{\mu}\lra&\gm\infty\\
\uparrow&&\ \uparrow\mbox{\tiny$\iota$}\\
K&\stackrel{v}\lra&\g\infty
\end{array}
\end{equation}

We say that $\mu/v$ is \emph{commensurable}, \emph{preserves the rank}, or \emph{increases the rank by one}, if the extension $\iota\colon \g\hk\g_\mu$ has this property, respectively.

By Theorem \ref{AllSmall}, this embedding $\iota$ is a small extension.
By Theorem \ref{MaxEqRk} and Lemma \ref{SemiUniv}, there is an embedding $j\colon \gm\hk\R^{I_S}_{\lx}$, for some $S\in\inii$, fitting into a commutative diagram:   
$$
\as{1.5}
\begin{array}{ccccc}
\kx&\stackrel{\mu}\lra&\gm\infty&&\\
\uparrow&&\ \uparrow\mbox{\tiny$\iota$}&\searrow\!\!\!\raise1.4ex\hbox{\tiny$j$\;}&\\
K&\stackrel{v}\lra&\g\infty&\stackrel{\ell}\lra&\R^{I_S}_{\lx}\infty
\end{array}
$$

Take $\nu$ to be the valuation on $\kx$ determined by the mapping $j\circ \mu$. Its value group is $\gn=j(\gm)\subset \R^{I_S}_{\lx}$, and the tautological isomorphism $j\colon \gm\ism\gn$ shows that the two valuations are equivalent.
This proves the following result.


\begin{proposition}\label{riiUniverse}
Every valuation on $\kx$ whose restriction to $K$ is equivalent to $v$ is equivalent to some $\rii$-valued valuation 
$$
\nu\colon \kx\lra \rii\infty
$$
such that $\ \gn\subset \rll$. If $\nu/v$ is commensurable, or preserves the rank, then $\ \gn\subset\gq$, or $\ \gn\subset\rlex$, respectively.
\end{proposition}

Therefore, we are led to find the equivalence classes of valuations in the subset
$$
\tz=\left\{\mu\colon \kx\to\rii\mid \mu\ \mbox{valuation},\ \gm\subset\rll\right\}
$$
of the set of all valuations with values in the fixed ordered group $\rii$.

This set $\tz$ has a natural structure of a tree. It admits a partial ordering:
$$
\mu\le\nu \sii \mu(f)\le\nu(f)\quad\mbox{for all }\ f\in\kx,
$$
and all intervals in $\tz$ are totally ordered.

For $\mu\in\tz$, denote by $\kpm\subset\kx$ the set of \emph{MacLane--Vaqui\'e key polynomials} for $\mu$. An element in $\kpm$ is a monic polynomial whose initial term generates a prime ideal in the graded algebra of $\mu$, which cannot be generated by the initial term of a polynomial of smaller degree. 

The subset of \emph{leaves} of $\tz$ (maximal elements) is characterized as follows:
$$
\ll(\tz)=\left\{\mu\in\tz\mid \kpm=\emptyset\right\}.
$$

Let $\tz^{\op{inn}}\subset\tz$ be the subtree of all inner nodes. For $\mu\in\tz^{\op{inn}}$, we define the \emph{degree} and the \emph{singular value} of $\mu$ as
$$
\deg(\mu)=\deg(\phi),\qquad \sval(\mu)=\mu(\phi),
$$
where $\phi\in\kpm$ is a key polynomial of minimal degree. It is well known that $\sval(\mu)$ is independent of the choice of $\phi$.

In order to classify the nodes of $\tz$ under equivalence of valuations, we establish on $\rll$ an equivalence relation. \e

\defn Two elements $x,y\in\rll$ are $\sme$-equivalent, and we write $x\sim_\sme y$, if there exists an isomorphism of ordered groups between the subgroups $\gen{\g,x}$ and $\gen{\g,y}$ which maps $x$ to $y$ and acts as the identity on $\g$.\e

With this equivalence at hand, we may characterize equivalence of valuations in $\tz$ as follows \cite[Prop. 6.3]{AGNR}.

\begin{proposition}\label{motivation}
Let $\mu,\nu\in\tz$ be two inner nodes. Then, $\mu\sim\nu$ if and only if the following three conditions hold:

(a) \ $\kpm=\kp(\nu)$.

(b) \ For all \,$a\in\kx\,$ such that \,$\deg(a)<\deg(\mu)$, we have $\,\mu(a)=\nu(a)$. 

(c) \ $\sval(\mu)\sim_\sme \sval(\nu)$.
\end{proposition}

Consider any subset $\gsme\subset\rll$ which is a faithful set of representatives of all $\sme$-classes in $\rll$. Consider the subtree 
$$
\tsim=\ll(\tz)\sqcup\left\{\mu \in\tz^{\op{inn}}\mid \sval(\mu)\in\gsme\right\}.
$$

Then, the following holds \cite[Thm. 7.1]{AGNR}.

\begin{theorem}\label{main2}
The mapping $\mu\mapsto[\mu]$ induces a bijection between $\tsim$ and the set of equivalence classes of valuations on $\kx$ whose restriction to $K$ is equivalent to $v$.
\end{theorem}
 
\subsection{Quasi-cuts. A concrete model for $\gsme$}\mbox{\null}\e

A \emph{quasi-cut} in $\gq$ is a pair $D=\left(D^L,D^R\right)$ of subsets such that $D^L\le D^R$ and $D^L\cup D^R=\gq$.
Then, $D^L$ is an initial segment of $\gq$, $D^R$ is a final segment of $\gq$ and $D^L\cap D^R$ consists of at most one element.

The set $\op{Qcuts}(\gq)$ of all quasi-cuts in $\gq$ admits a total ordering:
$$
D=\left(D^L,D^R\right)\le E=\left(E^L,E^R\right) \ \sii\ D^L\subset E^L\quad\mbox{and}\quad D^R\supset E^R.
$$
There is an embedding of ordered sets $\gq\hk\op{Qcuts}(\gq)$, which assigns to every $a\in\gq$ the \emph{principal quasi-cut} $\left(\left(\gq\right)_{\le a},\left(\gq\right)_{\ge a}\right)$.

The real models for the sets of cuts in $\gq$ that we discussed in Section \ref{secRealVS}, yield a quite concrete choice for the ordered set $\gsme$.

Inded, the equivalence reation $\sim_\sme$ is almost identical to the equivalence relation $\sim_\cut$. More precisely, for all $x\in\rll$ consider the quasi-cut $D_x$ in $\gq$, whose initial and final segments are:
$$
D_x^L=\{a\in\gq\mid a\le x\},\qquad D_x^R=\{a\in\gq\mid a\ge x\}.
$$

\begin{lemma}
For all $x,y\in\rll$, we have $x\sim_\sme y$ if and only if  $D_x=D_y$.
\end{lemma}

\begin{proof}
Take $x\in\gq$. Then, for all $y\in\rll$ we have
$$
x\sim_\sme y\ \sii\ x=y \ \sii\ D_x=D_y.
$$
The first equivalence follows from the fact that two subgroups $\d,\d'\subset\gq$ are $\g$-isomorphic as ordered groups, only when $\d=\d'$ and the isomorphism between them is the identity.

Finally, for $x,y\not\in\gq$, the lemma follows from Proposition \ref{Xd=cuts}, in the particular case $\d=\g$.
\end{proof}

As a consequence, we may take
$$
\gsme=\gq\sqcup \gnbc\sqcup \gbc,
$$
with the total ordering induced by that of $\rll$.
We derive a natural isomorphism of ordered sets:
$$\gsme\lra \op{Qcuts}(\gq),\qquad x\longmapsto D_x.$$  

As an immediate consequence of this isomorphism, we deduce that $\gsme$ is complete and $\gq$ is dense in $\gsme$, with respect to the order topology. Indeed, it is well known that the ordered set $\op{Qcuts}(\gq)$ has these properties.

\end{document}